\documentclass{amsart}

\usepackage{amsmath,amssymb,amsfonts,amstext,amsthm,bm}
\usepackage{url}
\usepackage{graphicx}
\usepackage{ragged2e}
\usepackage{float}
\usepackage{hyperref}
\usepackage{tikz}
\usepackage[all,cmtip]{xy} 
\usepackage{multicol}
\usepackage{verbatim}
\input xy
\xyoption{all}
\usepackage[all]{xy}

\newcommand{\nwc}{\newcommand}

\nwc{\aaa}{\mathcal{F}}
\nwc{\aap}{\mathcal{F}_{P}}
\nwc{\al}{\alpha}

\nwc{\C}{\mathbb{C}}
\nwc{\cb}{\overline{C}}
\nwc{\ccc}{\mathfrak{c}}
\nwc{\ch}{\widehat{C}}
\nwc{\cin}{\textbf{(v)}}
\nwc{\cl}{C'}
\nwc{\cod}{{\rm cod}}
\nwc{\cp}{\mathcal{C}_{P}}
\nwc{\cpll}{\mathfrak{c}_{P'}}
\nwc{\ct}{\widetilde{C}}

\nwc{\dd}{\mathcal{L}}
\nwc{\ddd}{\mathfrak{d}}
\nwc{\ddl}{\mathcal{L}'}
\nwc{\dlp}{\delta_{P}}
\nwc{\doi}{\textbf{(ii)}}

\nwc{\enq}{$$}

\nwc{\fl}{\flushleft}
\nwc{\fff}{\mathcal{F}}
\nwc{\ffp}{\mathcal{F}_{P}}
\nwc{\ffq}{\mathcal{F}_{Q}}
\nwc{\ffl}{\mathcal{F}'}

\nwc{\G}{\mathcal{G}}
\nwc{\Ga}{\Gamma}
\nwc{\gtl}{\widetilde{g}}
\nwc{\gon}{{\rm gon}}

\nwc{\hra}{\hookrightarrow}
\nwc{\hua}{h^{1}(C,\aaa )}

\nwc{\kk}{{\rm K}}

\nwc{\llb}{\mathcal{L}}

\nwc{\mb}{\mathbb}
\nwc{\mc}{\mathcal}
\nwc{\mm}{\mathfrak{m}}
\nwc{\mmp}{\mathfrak{m}_{P}}
\nwc{\mpd}{\mathfrak{m}_{P}^{2}}

\nwc{\nn}{\mathbb{N}}

\nwc{\ob}{\overline{\mathcal{O}}}
\nwc{\obr}{\mathcal{O}^*}
\nwc{\obp}{\overline{\mathcal{O}}_P}
\nwc{\och}{\mathcal{O}_{\hat{C}}}
\nwc{\oh}{\hat{\mathcal{O}}}
\nwc{\ohp}{\hat{\mathcal{O}}_{P}}
\nwc{\ol}{\mathcal{O}'}
\nwc{\oma}{\Omega (\mathfrak{a})}
\nwc{\omo}{\Omega (\mathcal{O})}
\nwc{\oo}{\mathcal{O}}
\nwc{\op}{\mathcal{O}_P}
\nwc{\opc}{\mathcal{O}_{P,C}}
\nwc{\oph}{\hat{\mathcal{O}}_{P}}
\nwc{\opl}{\mathcal{O}_{P}'}
\nwc{\oplc}{\mathcal{O}_{P,C}'}
\nwc{\opll}{\mathcal{O}_{P'}}
\nwc{\opt}{\tilde{\mathcal{O}}_{P}}
\nwc{\optt}{{\mathcal{O}}_{\tilde{P}}}
\nwc{\oq}{\mathcal{O}_{Q}}
\nwc{\oqt}{\tilde{\mathcal{O}}_{Q}}
\nwc{\ot}{\widetilde{\mathcal{O}}}
\nwc{\overop}{\bar{\oo}_{P}}

\nwc{\pb}{\overline{P}}
\nwc{\pbb}{P^*}
\nwc{\pbi}{\overline{P_{i}}}
\nwc{\pbr}{\overline{P_{r}}}
\nwc{\pgmd}{\mathbb{P}^{g+2}}
\nwc{\pgmu}{\mathbb{P}^{g+1}}
\nwc{\ph}{\hat{P}}
\nwc{\pp}{\mathbb{P}}
\nwc{\prv}{\noindent\textbook{Proof}:}
\nwc{\pt}{\widetilde{P}}
\nwc{\ptl}{\tilde{P}}
\nwc{\pum}{\mathbb{P}^{1}}

\nwc{\qh}{\hat{Q}}
\nwc{\qtl}{\tilde{Q}}
\nwc{\qua}{\textbf{(iv)}}

\nwc{\ra}{\rightarrow}
\nwc{\rh}{\hat{R}}

\nwc{\sei}{\textbf{(vi)}}
\nwc{\sep}{\beq\ast\ \ast\ \ast\enq}
\nwc{\sig}{\sigma}
\nwc{\Sig}{\Sigma}
\nwc{\ssp}{S_{P}}
\nwc{\sss}{{\rm S}}
\nwc{\sys}{\mathcal{L}}

\nwc{\tre}{\textbf{(iii)}}

\nwc{\um}{\textbf{(i)}}

\nwc{\vpb}{v_{\overline{P}}}
\nwc{\vtxp}{\widetilde{V}_{x,P}}
\nwc{\vxp}{V_{x,P}}
\nwc{\vv}{\mathcal{W}}
\nwc{\vvp}{\mathcal{W}_{P}}
\nwc{\val}{\mathcal{V}}

\let \wt=\widetilde

\nwc{\wh}{\hat{\omega}}
\nwc{\whp}{\hat{\omega}_{P}}
\nwc{\woch}{\omega\cdot\mathcal{O}_{\hat{C}}}
\nwc{\woh}{\omega\cdot\hat{\mathcal{O}}}
\nwc{\ww}{\omega}
\nwc{\wwb}{\omega^*}
\nwc{\wwct}{\omega _{\widetilde{C}}}
\nwc{\wwh}{\widehat{\omega}}
\nwc{\wwhp}{\widehat{\omega}_P}
\nwc{\wwp}{\omega _{P}}
\nwc{\wwt}{\widetilde{\omega}}
\nwc{\wwtp}{\widetilde{\omega}_P}

\nwc{\zz}{\mathbb{Z}}

\newtheorem{coro}{Corollary}[section]
\newtheorem{conv}[coro]{Convention}
\newtheorem{dfn}[coro]{Definition}

\newtheorem{lemma}[coro]{Lemma}

\newtheorem{prop}[coro]{Proposition}

\newtheorem{rem}[coro]{Remark}

\newtheorem{thm}[coro]{Theorem}
\newtheorem{conj}[coro]{Conjecture}
\newtheorem{ex}[coro]{Example}
\let \fl=\flushleft

\let \ga=\gamma
\let \sub=\subset
\let \be=\beta
\let \de=\delta
\let \al=\alpha
\let \pr=\prime
\let \la=\lambda
\let \ka=\kappa

\begin{document}

\title{Severi dimensions for unicuspidal curves}

\author{Ethan Cotterill}
\address{Instituto de Matem\'atica, UFF
Rua M\'ario Santos Braga, S/N,
24020-140 Niter\'oi RJ, Brazil}
\email{cotterill.ethan@gmail.com}

\author{Vin\'icius Lara Lima}
\address{Departamento de Matem\'atica, ICEx, UFMG
Av. Ant\^onio Carlos 6627,
30123-970 Belo Horizonte MG, Brazil}
\email{viniciuslaralima@gmail.com}

\author{Renato Vidal Martins}
\address{Departamento de Matem\'atica, ICEx, UFMG
Av. Ant\^onio Carlos 6627,
30123-970 Belo Horizonte MG, Brazil}
\email{renato@mat.ufmg.br}



\begin{abstract}
We study parameter spaces of linear series on projective curves in the presence of unibranch singularities, i.e. {\it cusps}; and to do so, we stratify cusps according to value semigroup. We show that {\it generalized Severi varieties} of maps $\mathbb{P}^1 \ra \mathbb{P}^n$ with images of fixed degree and arithmetic genus are often {\it reducible} whenever $n \geq 3$. We also prove that the Severi variety of degree-$d$ maps with a hyperelliptic cusp of delta-invariant $g \ll d$ is of codimension at least $(n-1)g$ inside the space of degree-$d$ holomorphic maps $\mathbb{P}^1 \ra \mathbb{P}^n$; and that for small $g$, the bound is exact, and the corresponding space of maps is the disjoint union of unirational strata. Finally, we conjecture a generalization for unicuspidal rational curves associated to an {\it arbitrary} value semigroup. 
\end{abstract}

\maketitle
\tableofcontents

\section{Introduction}
Rational curves are essential tools for classifying complex algebraic varieties. It is less well-known, however, that {\it singular} rational curves in projective space are often interesting in and of themselves. Rational curves of fixed degree $d$ in $\mb{P}^n$ are parameterized by an open subset of the Grassmannian $\mb{G}(d,n)$; and {\it singular} rational curves arise from special intersections of $(d-n-1)$-dimensional projection centers of rational normal curves with elements of their osculating flags in particular points. Precisely how this arises is often obscure, and in general a singularity is not uniquely determined by the {\it ramification data} encoded by these intersection numbers. One of the aims of this paper is to shed light on the conditions {\it beyond ramification} that determine a unibranch singularity, and in the process produce a relatively explicit description of the associated parameter spaces of unicuspidal rational curves. 

\medskip
The situation when $n=2$ has been studied extensively, and in this case the geometry is relatively well-behaved. Zariski \cite{Z} first established an upper bound for the dimension of any given component of the {\it Severi variety} $M^2_{d,g}$ of plane curves of fixed degree $d$ and genus $g$, and showed that whenever the upper bound is achieved, a general curve in that component is nodal. Zariski's result subsequently played an important role in Harris' celebrated proof \cite{H1} of the irreducibility of $M^2_{d,g}$; an upshot of their work is that every curve indexed by a point of $M^2_{d,g}$ lies in the closure of the irreducible sublocus of $M^2_{d,g}$ that parameterizes {\it $g$-nodal} rational curves.
However irreducibility, and the dominating role of the nodal locus, simultaneously fail in a particularly simple way when one replaces $M^2_{d,g}$ by $M^n_{d,g}$, the ``Severi variety" of degree-$d$ morphisms $\mb{P}^1 \ra \mb{P}^n$ of degree $d$ and arithmetic genus $g$. Indeed, as we saw in \cite{CFM1}, it is easy to construct examples of Severi varieties $M^n_{d,g}$ with components of strictly-larger dimension than that of the $g$-nodal locus as soon as $n \geq 8$. Each of the excess components produced in \cite{CFM1} parameterizes rational unicuspidal curves for which the corresponding value semigroup is of a particular type, which we christened {\it $\ga^{\ast}$-hyperelliptic} by analogy with Fernando Torres' $\ga$-hyperelliptic semigroups \cite{To2}.

\medskip
In this paper, we take a closer look at rational unicuspidal curves whose value semigroup $\sss$ is $\ga$-hyperelliptic. The simplest case is that of $\ga=0$, in which the underlying cusps are {\it hyperelliptic}, meaning simply that $2 \in \sss$. We show that when $d \gg g$, a hyperelliptic cusp of genus $g$ imposes at least $(n-1)g$ independent conditions on rational curves of degree $d$ in $\mb{P}^n$; moreover, we expect this lower bound to be sharp. This result should be compared against the benchmark codimension $(n-2)g$ of the space of degree-$d$ rational curves with $g$ (simple) nodes. Our analysis is predicated on a systematic implementation of a scheme for counting conditions associated with cusps described in \cite{CFM1}, for which we also give a graphical interpretation at the level of the Dyck path of the corresponding value semigroup. More precisely, our strategy is to fix the local {\it ramification profile}, making a linear change of basis if necessary so that the parameterizing functions of our rational curves are ordered according to their vanishing orders in the preimage of the cusp. We can then write down explicitly those conditions beyond ramification that characterize the cusp, and the upshot of this is an explicit dominant rational map from an affine space to each stratum of fixed ramification profile; in particular, each of these strata is unirational. 

\medskip
Going beyond the hyperelliptic case, we give an explicit lower bound for the codimension of the space $M^n_{d,g;{\rm S}_{\ga}}$ of unicuspidal rational curves of fixed degree $d \gg g$ with a $\ga$-hyperelliptic cusp of genus $g$ and of {\it maximal weight}, inside the space $M^n_d$ of all degree-$d$ rational  curves in $\mb{P}^n$. Torres showed that when $g \gg \ga$, such cusps are precisely those with value semigroup ${\rm S}_{\ga}=2\langle 2,2\ga+1 \rangle+\langle 2g-4\ga+1 \rangle$. We conjecture that our bound computes the {\it exact} codimension of $M^n_{d,g;{\rm S}_{\ga}}$ in $M^n_d$, and we give some computational as well as qualitative evidence for this. Motivated by our results for $\ga$-hyperelliptic cusps of maximal weight, we also give a conjectural combinatorial formula for the codimension of the locus $M^n_{d,g;{\rm S}}$ of rational curves with cusps of {\it arbitrary} type ${\rm S}$. The existence of such a combinatorial formula, albeit conjectural, aligns with the basic mantra (which we borrow from the study of compactified Jacobians of cusps) that the topology of $M^n_{d,g;{\rm S}}$ is controlled by ${\rm S}$ itself. It is also of practical utility. Indeed, we leverage this formula to obtain many new examples of unexpectedly-large Severi varieties associated with $\ga$-hyperelliptic value semigroups ${\rm S}$ of {\it minimal} weight.

\medskip
The smaller $n$ is, however, the more difficult it becomes to produce Severi varieties $M^n_{d,g;{\rm S}}$ of codimension strictly less than $(n-2)g$. Breaking this impasse forced us to rethink our basic organizational protocol for unicuspidal rational curves; and to focus on their stratification according to ramification profile as opposed to value semigroup. Indeed, any Severi variety $M^n_{d,g;{\rm S}}$ of codimension strictly less than $(n-2)g$ necessarily contains a {\it generic ramification stratum} with the same property. Accordingly, understanding the behavior of the value semigroup ${\rm S}$ attached to a {\it generic} parameterization with given ramification profile is crucial. We show that whenever $d \gg g$ and $n \geq 3$, the Severi varieties $M^n_{d,g;{\rm S}}$ obtained from generic cusps are nearly always of unexpectedly small codimension whenever their ramification profiles comprise sequences of consecutive even numbers. We anticipate that an approach based on limit linear series \cite{EH2} will show that the same codimension estimates remain operative when we substitute $M^n_{d,g;{\rm S}}$ by the space of linear series of degree $d$ and rank $n$ on a {\it general} curve of arbitrary genus whose images have a cusp of type ${\rm S}$.

\medskip

\subsection{Conventions}
We work over $\mb{C}$. By {\it rational curve} we always mean a projective curve of geometric genus zero; at times, it will be convenient to conflate a curve with a morphism that describes its normalization. A {\it cusp} is a unibranch (curve) singularity. We denote by $M^n_d$ the space of nondegenerate morphisms $f: \mb{P}^1 \ra \mb{P}^n$ of degree $d>0$. Here each morphism is identified with the set of coefficients of its homogeneous parameterizing polynomials, so $M^n_d$ is a space of frames over an open subset of $\mb{G}(n,d)$. We denote by $M^n_{d,g} \subset  M^n_d$ the subvariety of morphisms whose images have arithmetic genus $g>0$. These curves are necessarily singular. Clearly, $M^n_{d,g}$ contains all curves with $g$ simple nodes or $g$ simple cusps. 

\medskip
In this paper, we will invoke a number of standard tools from linear series and singularities. Accordingly, let $P\in C:=f(\pum)\subset\mathbb{P}^n$ be a cusp. 
Near $P$, the morphism $f$ 
is prescribed by a map $f: t \mapsto (f_1(t), \dots, f_n(t))$ of power series, or equivalently, by a map of rings
\begin{gather*}
\begin{matrix}
\phi : & R:= \mb{C}[[x_1,\dots,x_n]] & \longrightarrow & \mb{C}[[t]]\\
         &x_i            & \longmapsto     &  f_i(t).
\end{matrix}
\end{gather*}
Let $v_t:\mb{C}[[t]] \ra \mb{N}$ denote the standard valuation induced by the assignment $t \mapsto 1$. Let ${\rm S}:=v_t(\phi(R))$ denote the numerical {\it value semigroup} of $P$. The $t$-adic valuation $v_t$ computes the vanishing order in $P$ of elements of the local algebra of the cusp, so hereafter we will refer to {\it $t$-valuations} and {\it $P$-vanishing orders} interchangeably. The (local) {\it genus} of the singularity at $P$ is $\delta_P:=\#(\mb{N}\setminus{\rm S})$, and the (global arithmetic) genus of $C$ is the sum of all of these local contributions:
$$
g=\sum_{P\in C}\delta_P.
$$
We focus exclusively on {\it unicuspidal} rational curves, whose singularities are unibranch singletons. Abusively, we will use $M^n_{d,g; {\rm S}}$ to refer to the subvariety of $M^n_{d,g}$ that parameterizes unicuspidal genus-$g$ rational curves with value semigroup ${\rm S}$. The variety $M^n_{d,g; {\rm S}}$ is further stratified according to the 
strictly-increasing sequence of vanishing orders ${\bf k}=(k_0,k_1,\dots,k_n)$ in $t=0$ of linear combinations of the local sections $1,f_1(t),\dots,f_n(t)$ that parameterize the cusp. Hereafter, we always assume that $k_0=0$. Specifying these vanishing orders is equivalent to specifying a {\it ramification profile} ${\bf k}-(0,1,\dots,n)$ that measures their deviation relative to the generic sequence. As a matter of convenience, we will abusively conflate these two notions. Accordingly, we let $M^n_{d,g; {\rm S},{\bf k}} \sub M^n_{d,g; {\rm S}}$ denote the subvariety of unicuspidal curves with ramification profiles ${\bf k}$ in the preimages of their respective cusps.

\medskip
It will often be convenient to think of the genus of a cusp as an invariant of the associated numerical semigroup ${\rm S}$. Similarly, the {\it weight} of a cusp is defined in terms of the associated value semigroup by
\[
W= W_{\sss}:= \sum_{i=1}^g \ell_i- \binom{g+1}{2}
\]
where $\ell_1< \dots< \ell_g$ denote the $g$ elements of $\mb{N}\setminus{\sss}$.

\medskip
Given a nonnegative integer $\ga$, a numerical semigroup $\sss$ is {\it $\ga$-hyperelliptic} if it contains exactly $\ga$ even elements in the interval $[2,4\ga]$, and $4\ga+2 \in \sss$. Note that when $\ga=0$, the first condition is vacuous, while the second condition stipulates that $2 \in \sss$: in this situation, $\sss$ is simply {\it hyperelliptic}. A useful fact is that every numerical semigroup is $\ga$-hyperelliptic for a unique value of $\ga$, i.e., numerical semigroups are naturally stratified according to {\it hyperellipticity degree}.

\subsection{Roadmap}
A more detailed synopsis of the material following this introduction is as follows. In Section~\ref{hyperelliptic}, we prove Theorem~\ref{unibranch_theorem}, which gives a lower bound on the codimension of the locus of rational curves $f=(f_i)_{i=0}^n: \mb{P}^1 \ra \mb{P}^n$ with hyperelliptic cusps as a function of the vanishing orders of the parameterizing functions $f_i$ in the cusps' preimages. Theorem~\ref{unibranch_theorem} implies that the codimension of the locus of curves with a hyperelliptic cusp of genus $g$ is at least $(n-1)g$; to prove it, we produce an explicit packet of polynomials in the $f_i$ that impose independent conditions on their coefficients. Roughly speaking, these polynomials are of the simplest possible type suggested by the arithmetic structure of the value semigroup ${\rm S}=\langle 2,2g+1 \rangle$. Proposition~\ref{thmut2} establishes that whenever $g \leq 7$, these polynomials generate {\it all} nontrivial conditions imposed by a hyperelliptic cusp of genus $g$, and therefore, that the space of rational curves with hyperelliptic cusps has codimension {\it exactly} $(n-1)g$ and that each of its subsidiary fixed-ramification strata is unirational in this regime. Our argument is computer-based; however, to obtain a result for all $g$, it would suffice to prove that the pattern detailed in Table~\ref{G_and_H_polys} persists in general.

\medskip
In Section~\ref{gamma_hyperelliptic} we turn our focus to (rational curves with) $\gamma$-hyperelliptic cusps of maximal weight, which naturally generalize the hyperelliptic cusps considered in Section~\ref{hyperelliptic}. In Theorem~\ref{thmgam}, we obtain an explicit lower bound for the codimension of rational curves with genus-$g$ $\ga$-hyperelliptic cusps of maximal weight, whenever $g \gg \ga$. We use the arithmetic structure of the underlying semigroup ${\rm S}_{\ga}$ to produce an explicit packet of polynomials in the parameterizing functions $f_i$ of our curves, which in turn impose independent conditions on the coefficients of the $f_i$. We conjecture that these conditions are in fact a {\it complete} set of conditions imposed by $\ga$-hyperelliptic cusps of ${\rm S}_{\ga}$ type, and in Example~\ref{supporting_example} we provide evidence for this; see especially Table 2. 
Our analysis leads directly to Conjecture~\ref{gamma_hyp_conj}, which gives a value-theoretic prediction for the codimension of $M^n_{d,g;{\rm S}}$ in general, whenever this space is nonempty. Our Theorem~\ref{weak_bound} establishes that the prediction made by Conjecture~\ref{gamma_hyp_conj} for the codimension is at least a lower bound. 

\medskip
In Subsection~\ref{minimal_weight_semigroups}, we study (unicuspidal rational curves with) $\ga$-hyperelliptic value semigroups ${\rm S}$ of {\it minimal} weight. These include, in particular, the $\ga^{\ast}$-hyperelliptic examples studied in \cite{CFM1}. In Proposition~\ref{sporadic_minimal_weight_cases}, we exhaustively classify the minimally-ramified strata of such mapping spaces when the target dimension $n$ is at most 7, and as a result we find twenty-one new Severi varieties which should be unexpectedly large; we are able to verify this with Macaulay2 in thirteen cases by certifying that our set of conditions is exhaustive, before running out of computing power. These include the first-known examples with six- and seven-dimensional targets. Assuming the validity of Conjecture~\ref{gamma_hyp_conj}, in Proposition~\ref{new_Severi-excessive_families} we produce new infinite families of unexpectedly large Severi varieties in every target dimension (resp., genus) 6 (resp., 21) or larger. In Subsection~\ref{generic_semigroups}, on the other hand, we {\it unconditionally} construct infinitely many unexpectedly large Severi varieties from {\it generic} cusps with ramification profiles of the form $(2m, 2m+2, \dots, 2m+2n-2)$ whenever $n \geq 3$. In particular, our results are optimal in the ambient target dimension. At present we are not able to compute the exact genera of such Severi varieties when $n \geq 4$; see Theorem~\ref{generic_semigroups_n_geq_4} for a precise statement. However, when $n=3$ we manage to determine the corresponding {\it generic semigroup} (and its genus) explicitly; our Theorem~\ref{generic_semigroup_n=3} shows, in particular, that its genus growth is quadratic in $m$.

\subsection{Acknowledgements} We are grateful to Dori Bejleri, Nathan Kaplan, Nathan Pflueger, and Joe Harris for helpful conversations, to the anonymous referee for helpful comments on the exposition, and to Fernando Torres both for initiating the geometric study of the semigroups that bear his name and for his interest in this ongoing project. We dedicate this paper to his memory. The third named author is supported by CNPq grant 305240/2018-8.

\section{Counting conditions imposed by hyperelliptic cusps}\label{hyperelliptic}

Cusps form a naturally distinguished (simple) class of singularities. Accordingly, it makes sense to ask for dimension estimates for rational curves with at-worst cusps as singularities. In this section, we prove the following result for unicuspidal rational curves, when the cusps in question are hyperelliptic.

\begin{thm}\label{unibranch_theorem}
Given a vector ${\bf k}:=(k_0,\ldots,k_n)\in\mathbb{N}_{\geq 0}^{n+1}$, let $\mathcal{V}_{\bf{k}}:=M^n_{d,g;\langle 2,2g+1 \rangle, {\bf k}} \subset M^n_{d,g}$. 
Suppose, moreover, that $n \leq 2g$ and $d \geq \max(n,2g-2)$, and that $\mc{V}_{\bf k} \neq \emptyset$; then
\[
{\rm cod}(\mathcal{V}_{{\bf k}},M^n_d) \geq (n-1)g+\sum_{i=2}^n\left(\frac{k_i}{2}- i\right).
\]
In particular, the variety $\mathcal{V}:=M^n_{d,g;\langle 2,2g+1 \rangle}$ of rational curves with a unique singularity of hyperelliptic cuspidal type is of codimension at least $(n-1)g$ in $M^n_d$.
\end{thm}

\begin{rem} 
\emph{The condition $d \geq n$ is imposed by the requirement that our rational curves be nondegenerate, while the condition $n \leq 2g$ is an artifact of our method of proof (though likely this assumption may be removed). 
It is less clear what a reasonable lower threshold for the degree as a function of the genus should be; however, the assumption that $d \geq 2g-2$ includes the (canonical) case in which $d=2g-2$ and $n=g-1$. Note that $\mc{V}_{\bf k}$ is nonempty if and only if $k_0=0$, $k_1=2$, and the remaining $k_i$, $i=2,\dots,n$ belong to $\langle 2,2g+1\rangle$.}
\end{rem}

\begin{proof}
Let $C$ denote the image of a morphism $f:\mb{P}^1 \ra \mb{P}^n$ corresponding to a point of $\mc{V}$. Then $C$ ramifies at $P$ to order 
\begin{equation}\label{ramification}
r_P= \sum_{i=1}^n (k_i-i)
\end{equation}
and we have
\begin{equation}\label{codimension}
{\rm cod}(\mathcal{V}_{\bf k},M^n_d) = r_P+b_P-1
\end{equation}
in which $b_P$ denotes the number of independent conditions beyond ramification, and the $-1$ on the right-hand side of \eqref{codimension} arises from varying the preimage of $P$ along $\mb{P}^1$.

\medskip
Here we may assume $k_0=0$ and $k_1=2$ without loss of generality. 
In view of \eqref{ramification} and \eqref{codimension}, it suffices to show that each $f_i$, $i=2,\dots,n$ produces at least $g- \frac{k_i}{2}$ conditions beyond ramification. We may further suppose that $k_i<2g$, $i=1, \dots, n$, since whenever $k_i\geq 2g$ for some $i$, each of the $f_j$'s with $j \geq i$ produces at least $g$ ramification conditions. In light of our assumption that $n \leq 2g$ this means, in particular, that every $k_i$ is {\it even}.

\medskip
Without loss of generality, we may also assume that $f^{-1}(P)=(0:1) \in \mb{P}^1$, and that the cusp supported in $P$ is parameterized by $t$-power series $(f_1,\dots,f_n)$, where
\begin{equation}\label{power_series_parameterization}
\begin{split}
f_1&=t^2+a_{1,3}t^3+a_{1,4}t^4+ \cdots \\
f_2&=t^{k_2}+a_{2,k_2+1}t^{k_2+1}+a_{2,k_2+2}t^{k_2+2}+ \cdots \\
     &\ \ \ \ \ \ \vdots\\
f_n &=t^{k_n}+a_{n,k_n+1}t^{k_n+1}+a_{n,k_n+2}t^{k_n+2}+ \cdots
\end{split} 
\end{equation}
for suitable complex coefficients $a_{i,j}$. Note that the power series $f_i$ in \eqref{power_series_parameterization} is equal to the quotient of the $i$th and $0$th {\it global} parameterizing functions introduced previously.

\medskip
We now recursively define
\[
F_i:=f_i-f_1^{\frac{k_i}{2}}, F^{\ast}_{i,1}:= F_i, \text{ and }
F^{\ast}_{i,j}:=F^{\ast}_{i,j-1}-([t^{k_i+2(j-1)}]F^{\ast}_{i,j-1})f_1^{\frac{k_i}{2}+j-1}
\]
for every $2\leq i\leq n$ and $2\leq j\leq g-k_i/2$. The odd number $k_i+2j-1$ is a gap of the value semigroup of the hyperelliptic cusp, so the coefficient of $[t^{k_i+2j-1}]F^{\ast}_{i,j}$ must vanish. Let $C_{i,j}$ denote the polynomial in the coefficients of $f_i$ and $f_1$ associated with the vanishing condition $[t^{k_i+2j-1}]F^{\ast}_{i,j}=0$. Those coefficients of $f_i$ that appear in $C_{i,j}$ run from $a_{i,k_i+1}$ to $a_{i,k_i+2j-1}$; and $C_{i,j}$ is linear in the variables $a_{i,k_i+2j-1}$. It follows that the equations $C_{i,j}=0$ are algebraically independent; and for every $i \geq 2$, there are $g-\frac{k_i}{2}$ independent conditions beyond ramification, as required.
\end{proof}

\begin{ex}\label{exept1}
\emph{In Theorem~\ref{unibranch_theorem}, we showed that $b_P\geq \sum_{i=2}^n (g-i)$. 
In this example we detail the case in which $n=4$ and $g=7$, in order to show that our lower bound on $b_P$ is often an equality. It will motivate 
the following two results, and will serve as a template for the proofs of several others. Our strategy is predicated on producing a set of 
elements in the local algebra of a hyperelliptic cusp 
that algebraically generate {\it all} conditions beyond ramification; we call these \emph{$G$-polynomials}. To prove that $G$-polynomials algebraically generate the set of conditions beyond ramification, we begin with a ``universal" element $F$ in the local algebra, thought of as a polynomial in those $n$-tuples of power series $(f_1,\dots,f_n)$ that parameterize elements of $M^n_{d,g; {\rm S},{\bf k}}$ near their respective cusps. We then argue inductively on the $t$-adic valuations of these power series; as the valuation of $F$ necessarily belongs to ${\rm S}$, {\it gaps} of ${\rm S}$ enforce algebraic conditions on the coefficients of the parameterizing series $f_i$. 
}

\medskip
\noindent \emph{More precisely, let ${\rm S}=\langle 2,15 \rangle$ and ${\bf k}=(2,4,6,8)$. Near the cusp, the corresponding universal parameterization is defined by power series
\begin{align*}
f_1&=t^2+a_{1,3}t^3+a_{1,4}t^4+ \cdots \\
f_2&=t^4+a_{2,5}t^5+a_{2,6}t^6+ \cdots \\
f_3&=t^6+a_{3,7}t^7+a_{3,8}t^8+ \cdots \\
f_4&=t^8+a_{4,9}t^9+a_{4,10}t^{10}+ \cdots 
\end{align*}}

\noindent \emph{Now any $F\in\op$ is of the form
\begin{equation}
\label{equfff}
F=E+G+H+I
\end{equation}
where $E, F, G \in \C[f_1,\ldots,f_4]$ 
have $t$-adic valuations less than or equal to $2g-1$, and $I$ belongs to the conductor ideal of $\op$\footnote{It is well-known that every element of $\op$ with valuation $2g$ belongs to $I$, so a decomposition of $F$ as in \eqref{equfff} always exists.}. In particular, we have
\begin{align*}
E&:=\al_{(0)}+\al_{(1)}f_1; \\ 
G&:=\sum\limits^{4}_{i=2} \alpha_{(i)} f_i + \sum\limits_{j=2}^{6} \al_{(1^j)} f_1^{j}; \text{ and } \\
H &:=\al_{(2,1)} f_1 f_2+\al_{(2,1^2)} f_1^2 f_2 + \al_{(3,1)} f_1 f_3 + \al_{(2^2)} f_2^2 + \al_{(3,1^2)} f_1^2 f_3 + \al_{(2,1^3)} f_1^3 f_2 + \\
&\ \al_{(2^2,1)} f_1 f_2^2 + \al_{(4,1)} f_1 f_4 + \al_{(3,2)} f_2 f_3 + \al_{(2,1^4)} f_1^4 f_2 + \al_{(3,1^3)} f_1^3 f_3 + \al_{(4,1^2)} f_1^2 f_4 + \\ &\ \al_{(2^2,1^2)} f_1^2 f_2^2 + \al_{(3,2,1)} f_1 f_2 f_3  + \al_{(2^3)} f_2^3 + \al_{(4,2)} f_2 f_4 + \al_{(3^2)} f_3^2.
\end{align*}} 

\noindent \emph{
To iteratively produce algebraic constraints beyond ramification, we now argue as follows.  If $\al_{(0)}=0$ (resp. $\al_{(0)}=\al_{(1)}=0$) in (\ref{equfff}), we immediately deduce that $F$ has valuation at least $2$ (resp., $4$), as $1$ (resp., $3$)
belongs to $\mb{N} \setminus {\rm S}$ and as such is disallowed as a vanishing order. 
Similarly, terms in the conductor ideal contribute no conditions; so without loss of generality, we may assume that $E=I=0$.
Accordingly, we may rewrite $F$ as
\begin{align*}
F&:=\alpha_{(2)}f_2+\al_{(1^2)}f_1^2
\\
  &\ +\alpha_{(3)}f_3+\al_{(1^3)} f_1^3+\al_{(2,1)} \ f_1 f_2\\
 &\ +\alpha_{(4)} f_4+\al_{(1^4)} f_1^4+\al_{(2,1^2)} f_1^2 f_2+\al_{(3,1)} f_1 f_3+\al_{(2^2)} f_2^2 \\
  &\ +\al_{(1^5)}f_1^5+\al_{(3,1^2)} f_1^2 f_3+\al_{(2,1^3)}f_1^3 f_2+\al_{(2^2,1)}f_1 f_2^2+\al_{(4,1)}f_1 f_4 + \al_{(3,2)}f_2 f_3\\ 
 &\ +\al_{(1^6)}f_1^6+\al_{(2,1^4)} f_1^4 f_2+\al_{(3,1^3)} f_1^3 f_3 + \al_{(4,1^2)} f_1^2 f_4 + \al_{(2^2,1^2)} f_1^2 f_2^2+\al_{(3,2,1)}f_1 f_2 f_3 \\
 &\ \ \ \ \ \ \ \ \ \ \ \ \ \ \ \  \ \ \ \ \ \ \ \ \  \ \ \ \ \ \ \ \  \ \ \ +\al_{(2^3)}f_2^3+\al_{(4,2)} f_2 f_4+\al_{(3^2)}f_3^2
\end{align*}
in which each line in the above diagram corresponds to a homogeneous 
summand of fixed valuation.}

\medskip
\noindent \emph{We now 
expand $F$ as a power series in $t$. The initial part of the expansion reads
$$
F=(\alpha_{(2)} + \al_{(1^2)}) t^4+(a_{2,5}  \alpha_{(2)} + 2 a_{1,3}  \al_{(1^2)}) t^5+O(t^6)
$$
so that 
$v_t(F)>4$ if and only if
\[
\phi_1(\alpha_{(2)},\al_{(1^2)}):=\alpha_{(2)}+ \al_{(1^2)}=0.
\]
Accordingly, we set $\al_{(1^2)}=\ell_1(\alpha_{(2)}):=-\alpha_{(2)}$; and, therefore, the coefficient of $t^5$ above becomes $$\psi_1(\alpha_{(2)}):=\alpha_{(2)}(a_{2,5} - 2 a_{1,3}).$$
}\end{ex}

{\flushleft The} fact that $5$ is a gap now forces
$\psi_1=0$ for every $\al_{(2)} \in \mb{C}$, i.e., that
\[
G_1=0
\]
where $G_1:=a_{2,5} - 2 a_{1,3}$; up to multiplication by constant, this is the unique condition imposed by the fact that $5 \notin \sss$.

\medskip
{\flushleft In} order to capture the essence of the method, we will no longer explicitly record the polynomials but simply mention which coefficients $\alpha_{\lambda}$ are involved at each step.

{\flushleft In} a second step, we write
$$
F=\psi_1 t^5+ \phi_2 t^6+ \psi_2 t^7+ O(t^8)
$$ 
and we impose $\phi_2=0$, which in turn allows us to rewrite $\psi_2$ as
$$
\psi_2=\alpha_{(2)}G_2+\alpha_{(3)}G_3+\alpha_{(2,1)}G_1.
$$
The upshot is that through step two of our procedure, the polynomials $G$ are responsible for {\it all} of the conditions imposed on the coefficients of the parameterizing functions $f_i$. Indeed, $G_1=0$ is the unique condition arising from step one, so the new conditions in step two are $G_2=0$ and $G_3=0$. 

\noindent More generally, at step $s$ of our procedure, we write
$$
F=\psi_{s-1} t^{2s+1}+\phi_s t^{2s+2}+\psi_{s+1}t^{2s+3}+O(t^{2s+4})
$$ 
and set $\phi_s=0$; then $\psi_s=0$ is forced, and this allows us to rewrite $\psi_s$ as a linear combination of polynomials in the coefficients of the parameterizing functions $f_i$. 


\medskip
\noindent Referencing the induced decomposition $F=G+H$, we call those polynomial conditions produced by $G$ (resp., $H$) as we iterate our procedure {\it $G$-polynomials} (resp., {\it $H$-polynomials}). The following table gives a precise description of each of these sets of polynomials in the coefficients of the $f_i$; the upshot is that the $H$-polynomials are {\it contained} in the set of $G$-polynomials.

 \begin{table}[h]\label{G_and_H_polys}
\begin{tabular}{r|r|r|r|r|r|r|}
 & 
\multicolumn{3}{|c|}{G-polynomials} & 
\multicolumn{3}{|c|}{H-polynomials} \\ 
\hline                               
Gap & $\alpha_{(2)}$ & $\alpha_{(3)}$ & $\alpha_{(4)}$ & $\al_{(2,1)}$ & $\al_{(3,1)}$ & $\al_{(4,1)}$ \\
\hline    
$5$ & $G_1$ & & & & &  \\
$7$ & $G_2$ & $G_3$ & & $G_1$ & &  \\
$9$ & $G_4$ & $G_5$ & $G_6$ & $G_2$ & $G_3$ & \\
$11$ & $G_7$ & $G_8$ & $G_9$ & $G_4$ & $G_5$ & $G_6$ \\
$13$ & $G_{10}$ & $G_{11}$ & $G_{12}$ & $G_7$ & $G_8$ & $G_9$       
\end{tabular}
\vspace{10pt}
\caption{$G$- and $H$-polynomials for hyperelliptic cusps}
\end{table}

\noindent More precisely, the conditions imposed by the $G$'s (whose scalar multipliers are the variables $\alpha_{(2)}$, $\alpha_{(3)}$, $\alpha_{(4)}$) are systematically reproduced (at precisely one later step) by the $H$'s (whose scalar multipliers are $\al_{(2,1)}$, $\al_{(3,1)}$, $\al_{(4,1)}$). All available empirical evidence suggests that this same pattern persists for arbitrary values of $g$ and $n$; proving that this holds in general, which in turn would imply a version of Proposition~\ref{thmut2} without hypotheses on $g$, is an interesting problem.

\medskip
\noindent In Figure 1, we give a graphical interpretation of the conditions imposed by the $G$'s, i.e., by the polynomials $F_i=F^{\ast}_{i,1}$ and their inductively derived children $F^{\ast}_{i,j}$, $j \geq 2$. Graphically speaking, the index $i$ specifies a column, while the index $j$ specifies a number of upward steps from the Dyck path that codifies the semigroup. This graphical interpretation generalizes naturally to the case of $\ga$-hyperelliptic cusps, as we will see later (compare Figure 2 below).

\begin{figure}[H] \label{hyperelliptic_graphic}
\begin{center}
\includegraphics[width=5cm,height=5cm,keepaspectratio]{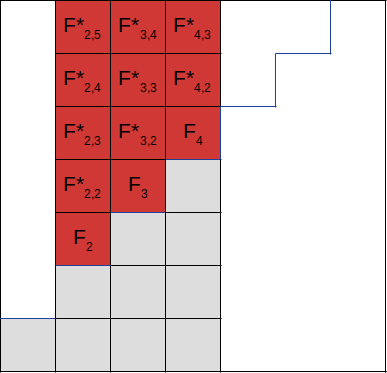}
\end{center}
\caption{Conditions contributing to $b_P$ and $r_P$ for rational curves with a hyperelliptic cusp when $g=7$ and $n=4$.}
\end{figure}

\begin{lemma}\label{lemleq}
The polynomials of the form $G=\sum\limits^{n}_{i=2} \alpha_{(i)} f_i + \sum\limits_{j=k_2/2}^{g-1} \al_{(1^j)} f_1^{j}$
impose exactly $
\sum_{i=2}^n \left(g-\frac{k_i}{2}\right)
$
independent conditions beyond ramification.
\end{lemma}

\begin{proof}
In the general case, when the parameterizing functions $f_i$ have arbitrary even $P$-vanishing orders, we write
{\small
\begin{align*}
f_1&=t^2+a_{1,3}t^3+a_{1,4}t^4+ \cdots \\
f_2&=t^{k_2}+a_{2,k_2+1}t^{k_2+1}+a_{2,k_2+2}t^{k_2+2}+ \cdots \\
f_3&=t^{k_3}+a_{3,k_3+1}t^{k_3+1}+a_{3,k_3+2}t^{k_3+2}+ \cdots \\
     &\ \ \ \ \ \ \vdots\\
f_n &=t^{k_n}+a_{n,k_n +1}t^{k_n +1}+a_{n+1,k_n+2}t^{k_n+2}+ \cdots
\end{align*}
}
\hspace{-3pt}and we set $l_i:=k_i/2$, $d_i:=l_{i+1}-l_i-1$ for all $2 \leq i \leq n-1$, and $d_n:=g-l_n-1$. As a shorthand, we will write $\al_i$ (resp., $\be_i$) in place of $\al_{(i)}$ (resp., $\al_{(1^i)})$ in what follows. We then have
{\small
\begin{align}
G&=(\alpha_{2} + \be_{l_2})t^{2l_2} +(\alpha_{2} a_{2,2l_2+1}  +\be_{l_2} l_2 a_{1,3})t^{2l_2+1}+\cdots \label{lab1} \\
  & \ \ \ \ \ + \beta_{l_2 +1} t^{2(l_2+1)} + \beta_{l_2 +1} (l_2+1) a_{1,3} t^{2(l_2+1)+1} + \cdots \label{lb1} \\
  &\ \ \ \ \ \ \ \ \ \ \ \ \ \ \ \ \ \ \ \ \  \ \ \ \  \ \vdots \nonumber\\
  & \ \ \ \ \  + \beta_{l_2+d_2}t^{2(l_2+d_2)} + \beta_{l_2 +d_2} (l_2+d_2) a_{1,3} t^{2(l_2+d_2)+1}+\cdots  \label{lb2} \\
   &\ \ \ \ \ \ \ \ + (\alpha_3 + \beta_{l_3})t^{2l_3} + (\alpha_3 a_{3,2l_3+1} +\beta_{l_3} l_3 a_{1,3})t^{2l_3+1}+\cdots   \label{lab2} \\
  &\ \ \ \ \ \ \ \ + \beta_{l_3 +1} t^{2(l_3+1)} + \beta_{l_3 +1} (l_3+1) a_{1,3} t^{2(l_3+1)+1} + \cdots \nonumber \\
  &\ \ \ \ \ \ \ \ \ \ \ \ \ \ \ \ \ \ \ \ \  \ \ \ \  \ \ \ \ \vdots \nonumber\\
  &\ \ \ \ \ \ \ \  + \beta_{l_3+d_3} t^{2(l_3+d_3)} + \beta_{l_3 +d_3} (l_3+d_3) a_{1,3} t^{2(l_3+d_3)+1} + \cdots \nonumber\\  
  &\ \ \ \ \ \ \ \ \ \ \ \ \ \ \ \ \ \ \ \ \  \ \ \ \  \ \ \ \ \ \ \ \vdots \nonumber\\
&\ \ \ \ \ \ \ \ \ \ \ \ \ \ \ \ \ \ \ \ \  \ \ \ \  \ \ \ \ \ \ \ \vdots \nonumber\\
 & \ \ \ \ \ \ \ \ \ \ + (\alpha_n + \beta_{l_n})t^{2l_n} + (\alpha_n a_{n,2l_n+1} +  
\beta_{l_n} l_n a_{1,3})t^{2l_n+1}+\cdots\nonumber \\
  &\ \ \ \ \ \ \ \ \ \  + \beta_{l_n +1} t^{2(l_n+1)} + \beta_{l_n +1} (l_n+1) a_{1,3} t^{2(l_n+1)+1}+\cdots  \nonumber\\
  &\ \ \ \ \ \ \ \ \ \ \ \ \ \ \ \ \ \ \ \ \  \ \ \ \  \ \ \ \ \ \ \ \ \ \ \nonumber\vdots \\
  &\ \ \ \ \ \ \  \ \ \ + \beta_{l_n+d_n} t^{2(l_n+d_n)} + \beta_{l_n +d_n} (l_n+d_n) a_{1,3} t^{2(l_n+d_n)+1}+\cdots \nonumber.
\end{align}
}

\medskip
{\fl The} gaps of the semigroup $\sss$ of $P$ determine conditions beyond ramification, inasmuch as they may never arise as $P$-vanishing orders of elements of $\op$. 
Our method consists in forcing vanishing to successively higher orders; the steps of the associated process are indexed by gaps of $\sss$. In the first step we start with the gap $2l_2+1$, which is the first gap that appears as an exponent in our putative expansion of $G$ with generic coefficients. Accordingly we choose 
\begin{equation}
\label{equbt2}
\beta_{l_2} = -\alpha_2
\end{equation}
in order to eliminate the term of degree $2l_2$. The term of degree $2l_2+1$ now must vanish, which forces
$$\alpha_2 G_1=0$$ 
where $G_1=a_{2,2l_2+1}- l_2 a_{1,3}$. In particular, $G_1=0$ is the unique condition enforced by the gap $2l_2+1$. 
Moreover, the coefficient of each monomial $t^k$ appearing in line (\ref{lab1}) is computed by a linear form $\ell_{1,k}(\alpha_2)$ in $\al_2$.

{\flushleft Similarly}, in a second step indexed by the gap $2(l_2+1)+1$, we set
\begin{equation}
\label{equba1}
\beta_{l_2+1} = -\ell_{1,2(l_2+1)}(\alpha_2)
\end{equation}
in order to eliminate the term of degree $2(l_2+1)$. Eliminating the term in degree $2(l_2+1)+1$ now forces
$$ \beta_{l_2 +1} (l_2+1) a_{1,3}+\ell_{1,2(l_2+1)+1}(\alpha_2)=0$$ 
which in turn implies that
$$
\alpha_2G_2 =0
$$
where $G_2$ is a polynomial on the coefficients of the (parameterizing functions of the) curve. So we obtain 1 additional condition, namely, $G_2=0$. Moreover, in light of (\ref{equba1}), we see that any coefficient of a monomial $t^k$ appearing in line (\ref{lb1}) is computed by a linear form $\ell_{2,k}(\alpha_2)$.

{\fl More generally,} each step through the $(d_2+1)$th (indexed by the gap $2(l_2+d_2)+1$) produces one additional independent condition; and the coefficients of the powers $t^k$ in every line until and including (\ref{lab2}) are computed by linear forms $\ell_{3,k}(\alpha_2),\ldots,\ell_{d_2+1,k}(\alpha_2)$.

{\fl At} the $(d_2+2)$th step, we set 
\begin{equation}
\label{equba2}
 \beta_{l_3}= -\alpha_3 - \sum_{i=1}^{d_2+1}\ell_{i,2l_3}(\alpha_2)
\end{equation}
in order to eliminate the term of degree $2l_3$. This, in turn, forces
$$ 
(\alpha_3 a_{3,2l_3+1} +\beta_{l_3} l_3 a_{1,3})+ \sum_{i=1}^{d_2+1}\ell_{i,2l_3+1}(\alpha_2)=0
$$ 
in order to obtain vanishing in degree $2l_3+1$. We now obtain
$$
\alpha_2 G_{d_2+2}+\alpha_3 G_{d_2+3} =0
$$
where $G_{d_2+2}$ and $G_{d_2+3}$ are polynomials in the coefficients of (the parameterizing functions of) the curve. So we get 2 additional conditions, namely $G_{d_2+2}=0$ and $G_{d_2+3}=0$. We see, moreover, that every coefficient of a monomial $t^k$ in line (\ref{lab2}) is computed by a linear form $\ell_{l_3,k}(\alpha_2,\alpha_3)$: indeed, any such coefficient depends a priori on $\alpha_3$ and $\beta_{l_3}$, but according to (\ref{equba2}), $\beta_{l_3}$ depends linearly on the $\alpha_2$ and $\al_3$.

{\fl Iterating our selection procedure} ultimately yields at most 
$\sum_{i=2}^{n}(i-1)(d_i+1)$ conditions, generated by the $G_i$'s at every step. Here
{\small
\[
\begin{split}
\sum_{i=2}^{n}(i-1)(d_i+1)&=\sum_{i=2}^{n-1}(i-1)(l_{i+1}-l_{i})+(n-1)(g-l_n)\\                                                                                &=\sum_{i=2}^{n-1} (i-1)\left(\frac{k_{i+1}-k_{i}}{2}\right)+(n-1)\left(g-\frac{k_n}{2}\right) \\ 
&=\sum_{i=2}^{n}\left(g-\frac{k_{i}}{2}\right).
\end{split}
\]
}

\noindent 
The argument above shows that the polynomials $G$ impose at most $\sum_{i=2}^n \left(g-\frac{k_i}{2}\right)$ algebraically independent conditions beyond ramification. On the other hand, the proof of Theorem~\ref{unibranch_theorem} establishes that there are at least this number of algebraically independent conditions, all imposed by the polynomials $F_i$ and $F_{i,j}^*$. But as the $F_i$ and $F_{i,j}^*$ are $G$-polynomials, we conclude that the $G$-polynomials impose \emph{exactly} $\sum_{i=2}^n \left(g-\frac{k_i}{2}\right)$ algebraically independent conditions beyond ramification.

\end{proof}

\justifying

\begin{prop}\label{thmut2}
Let $\mathcal{V}$ be the variety of rational curves with a unique singularity that is a hyperelliptic cusp. Suppose that $n \leq 2g$ and $d \geq \max(n,2g-2)$; then
\[
{\rm cod}(\mathcal{V},M^n_d) = (n-1)g
\]
and each fixed-ramification substratum $\mc{V}_{\bf k} \sub \mc{V}$ is unirational of codimension $(n-1)g+\sum_{i=2}^n\left(\frac{k_i}{2}- i\right)$ whenever $g \leq 7$.
\end{prop}

\begin{proof}
We verified using Macaulay2 that 
the set of $H$-polynomials is contained in the set of $G$-polynomials for every $3 \leq n \leq g-1$ whenever $4 \leq g \leq 7$; see \cite[ancillary file]{ArqCont}. This means, in turn, that the algebra of conditions imposed by hyperelliptic cusps is generated by the leading coefficients $C_{i,j}$ of the polynomials $F^{\ast}_{i,j}$ introduced in the proof of Theorem~\ref{unibranch_theorem}. The unirationality of $\mc{V}_{\bf k}$ now follows from the fact that each $C_{i,j}$ is linear in the variable $a_{i,k_i+2j-1}=[k_i+2j-1]f_i$ of the ``universal" parameterization $f$ with ramification profile ${\bf k}$.
\end{proof}

\begin{rem}
The area of the rectangle determined by columns 2 through $n$ of our Dyck diagram (of conditions contributing to $b_P$ and $r_P$) is precisely $(n-1)g$, and in our graphical interpretation all of the corresponding boxes are marked; cf. Figure 1.
\end{rem}

\section{Counting conditions imposed by $\gamma$-hyperelliptic cusps}\label{gamma_hyperelliptic}
In this section, using (the proof of) Theorem~\ref{unibranch_theorem} as a template, we establish a lower bound on the number of conditions imposed on rational curves by a $\ga$-hyperelliptic cusp of genus $g$ whose value semigroup is of maximal weight. Fernando Torres proved \cite{To2} that whenever $g \gg \ga$, the unique numerical semigroup with this property is ${\rm S}_{g,\ga}=\langle 4,4\ga+2,2g-4\ga+1 \rangle$.

\begin{thm}\label{thmgam}
Let $\mathcal{V}_{\sss_{g,\gamma}}:=M^n_{d,g;{\rm S}_{g,\ga}} \subset M^n_{d,g}$ denote the subvariety consisting of rational curves with a single singularity $P$ that is a $\gamma$-hyperelliptic cusp with value semigroup ${\rm S}_{g,\ga}$, $\ga>0$. Assume as before that $n \leq 2g$, $d \geq \max(2g-2,n)$ and, moreover, that 
$g \geq 4\ga+2$. Then
\[
\begin{split}
{\rm cod}(\mc{V}_{{\rm S}_{g,\ga}},M^n_d)  &\geq (n-1)g- \de_{n \leq \ga} (2\ga+n-j^{\ast \ast}-3)- \de_{n \geq \ga+1} (3\ga-j^{\ast \ast}-2) 
\end{split}
\]
where $\de$ is Dirac's delta and $j^{\ast \ast}$ is either the unique nonnegative integer for which $g \in (6\ga-2j^{\ast \ast}-1,6\ga-2j^{\ast \ast}+1]$ or else $j^{\ast \ast}=-1$.
\end{thm}

\begin{rem}
The hypothesis that $g \geq 4\ga+2$ is made in order to ensure that $8\ga+4 \leq 2g$, which slightly simplifies the exposition below. Note that $2g-4\ga+1>4\ga+2$, i.e. $g \geq 4\ga+1$, is automatic, because ${\rm S}_{g,\ga}$ is $\ga$-hyperelliptic by assumption.
\end{rem}

\begin{proof}
The analysis required to produce a lower codimension bound is more delicate than in the $\ga=0$ case, because of the structure of the underlying semigroup ${\rm S}_{g,\ga}$. We work locally near a $\ga$-hyperelliptic cusp $P$ of a curve $[C] \in \mc{V}_{{\rm S}_{g,\ga}}$ with $P$-vanishing order vector ${\bf k}=(k_0,\dots,k_n)$; that is, $[C]$ belongs to $\mc{V}_{{\bf k}}: =M^n_{d,g;{\rm S}_{g,\ga},{\bf k}}$.
Without loss of generality, we may assume $k_0=0$, $k_1=4$, and that $k_{j^{\ast}}=4\ga+2$ for some positive integer $j^{\ast} \leq \ga+1$. Abusively, hereafter we refer to the local incarnation of $f$ as $(f_1(t),\dots,f_n(t))$, in which $f_i=t^{k_i}+ \sum_{\ell \geq k_i+1} a_{i,\ell} t^{\ell}$ for all $i=1,\dots,n$, for some local coordinate $t$ centered in $P$.
The arithmetic structure of ${\rm S}_{g,\ga}$ interacts with the parameterization $f$ underlying $C$ via the following device.

\begin{dfn}\label{defred}
\emph{Given a distinct set of natural numbers $k_1,\dots,k_n$, a \emph{decomposition} of $s\in\mathbb{N}$ with respect to $k_1,\dots,k_n$ is an equation
\begin{equation}\label{decomposition}
s=m_1k_1+\ldots+m_{n}k_{n}
\end{equation}
with non-negative integer coefficients $m_j, j=1, \dots,n$. Its \emph{underlying partition} is $(k_1^{m_1},\dots, k_n^{m_n})$. 
A decomposition as in \eqref{decomposition} is \emph{reducible} whenever some proper sub-sum of the right-hand side of \eqref{decomposition} decomposes with respect to $k_1,\dots,k_n$;
otherwise it is \emph{irreducible}.} 
\end{dfn}

The following auxiliary notion will also be useful.
\begin{dfn}\label{defrho}
\emph{Given an element $s$ of a numerical semigroup $\sss$, we set
$$
\rho(s):=\#\{r>s\,|\, r\not\in\sss\}.
$$}
\end{dfn}

{\fl \bf Case 1:} {\it ${\bf k}$ consists entirely of even integers}. 
As in the $\ga=0$ case, we have
\[
{\rm cod}(\mathcal{V}_{\bf k},M^n_d) = r_P+b_P-1
\]
where 
$r_P=\sum_{i=1}^n (k_i -i)$ is the ramification of $f$, and $b_P$ is the number of independent conditions beyond ramification imposed by ${\rm S}_{g,\ga}$ on $f$. These conditions beyond ramification are induced by polynomials in the $f_i$ indexed by irreducible decompositions of elements $s \in {\rm S}_{g,\ga}$ with parts $k_1, \dots, k_n$. 

\medskip
{\bf Subcase 1.1: $s \leq 4\ga$.} Note that $s$ is a multiple of 4 in this range. When $s \notin \{k_2,\dots,k_{j^{\ast}-1}\}$, $s$ admits a unique irreducible decomposition with respect to $k_1,\dots, k_n$, whose underlying partition is $(4^{s/4})$; $s$ then gives zero net contribution to the codimension of $\mc{V}_{\bf k}$. Now say $s=k_j$ for some $j \in [2,j^{\ast}-1]$. Then $s$ contributes (at least) $\rho(s)$ independent conditions to $b_P$. To see this, we begin much as in the $\ga=0$ case by setting $F_j:= f_j-f_1^{\frac{k_j}{4}}$. Then $v_t(F_j)$ is at least $k_j+1$, which belongs to $\mb{N} \setminus {\rm S_{g,\ga}}$ and is thereby precluded.
By the same logic, we have
\begin{equation}\label{k_j_conditions}
[t^{k_j+1}]F_j= [t^{k_j+2}]F_j= \dots= [t^{k_j^{\ast}-1}]F_j=0
\end{equation}
where $k_j^{\ast}$ is the smallest element of ${\rm S}_{g,\ga}$ strictly greater than $k_j$. Independence of the linear vanishing conditions \eqref{k_j_conditions} is clear. On the other hand, once the conditions \eqref{k_j_conditions} have been imposed, we have $v_t(F_j)= k_j^{\ast}$, since the remaining nonzero coefficients of $F_j$ are generic. If $k_j<4\ga$, we now iterate this procedure, setting $F_j^{\ast}:= F_j- [t^{k_j^{\ast}}]F_j \cdot f_1^{\frac{k_j^{\ast}}{4}}$. Replacing $F_j$ by $F_j^{\ast}$ and $k_j$ by $k_j^{\ast}$ yields a set of vanishing conditions analogous to \eqref{k_j_conditions}. On the other hand, if $k_j=4\ga$, then $k_j^{\ast}=k_{j^{\ast}}= 4\ga+2$, and we set $F_j^{\ast}:= F_j- [t^{4\ga+2}]F_j \cdot f_{j^{\ast}}$, whose leading term must vanish. We then iterate by replacing $F_j$ by $F_j^{\ast}$ and subtracting a scalar multiple of any monomial in powers of $f_1$ and $f_{j^{\ast}}$ with valuation equal to that of (the new version of) $F_j$. Our procedure continues in this way until all gaps of ${\rm S}_{g,\ga}$ greater than $k_j$ and less than $2g-4\ga$ have been exhausted, and the conditions obtained are algebraically independent; indeed, the {\it linear} part of the condition imposed by a given gap $q \in \mb{N}_{>k_j} \setminus {\rm S}_{g,\ga}$ is precisely $a_{j,q}- \frac{k_j}{4} a_{1,q+4-k_j}$, so new variables appear {\it linearly} in the coefficients that are required to vanish at every step.

\medskip
Our iterative procedure may be interpreted graphically with respect to the Dyck path $\mc{P}$ associated with ${\rm S}_{g,\ga}$ inside its $g \times g$ bounding box. There is a horizontal step in $\mc{P}$ labeled by $k_j$; so $k_j$ singles out a column in the Dyck diagram, and $\rho(k_j)$ is the vertical distance to the top of that column. In particular, we have $\rho(k_j)=g-\frac{3}{4}k_j$; at first glance, it might seem natural to guess that the contribution of $s=k_j$ to $r_P+b_P$ is $g+ (\frac{1}{4}k_j-j)$. Note, in particular, that this contribution is at least $g$, with equality if and only if $k_j=4j$.

\medskip
While this is indeed a useful approximation it is not quite correct, as $2g-4\ga+1 \in {\rm S}_{g,\ga}$ is not realizable as a positive linear combination of $k_1, \dots, k_n$. The upshot of this is that it is impossible to continue inductively walking up the Dyck column indexed by $k_j$ simply by adding monomials in $f_1,\dots, f_n$ to $F_j$, $F_j^{\ast}$ at each stage, since no {\it monomial} in $f_1,\dots, f_n$ has valuation equal to $2g-4\ga+1$. Rather, in order to continue ascending the column indexed by $k_j$ ``past" $2g-4\ga+1$, it is necessary to leverage the other columns, and their inductively-constructed polynomials. We will return to this issue momentarily.

\medskip
{\bf Subcase 1.2: $4\ga+2 \leq s$ and $s \neq 8\ga+4$.} 
In this range, $s$ again admits either $2$ or $1$ distinct irreducible decompositions with respect to $k_1, \dots, k_n$, depending upon whether $s$ belongs to $\{k_{j^{\ast}+1},\dots,k_n\}$ or not. If $s \notin \{k_{j^{\ast}+1},\dots,k_n\}$, then $s$ gives zero net contribution to the codimension of $\mc{V}_{\bf k}$; so without loss of generality we may assume $s=k_j$ for some $j \in [j^{\ast}+1,n]$. Then $s$ has irreducible decompositions with underlying partitions $(k_j)$ and $(4^{\frac{k_j}{4}})$ (resp., $(4\ga+2,4^{\frac{k_j-(4\ga+2)}{4}})$) depending upon whether $k_j$ is divisible by 4 or not. Correspondingly, we define $F_j:= f_j-f_1^{\frac{k_j}{4}}$ (resp., $F_j:=f_j-f_{j^{\ast}} f_1^{{\frac{k_j-(4\ga+2)}{4}}}$). We now inductively ``walk" up the column of the Dyck diagram indexed by $k_j$ following the same inductive procedure as in Subcase 1.1. To a first approximation, it is useful to imagine that every gap of ${\rm S}_{g,\ga}$ strictly greater than $k_j$ imposes a condition that depends linearly on a previously-unseen variable, and there are $\rho(k_j)$ of these. The precise value of $\rho(k_j)$ depends on how large $k_j$ is relative to $2g-4\ga+1$; writing $k_j=4\ga+2\ell$ for some $\ell \geq 0$, we have

\[
\rho(k_j)= (g-3\ga-\ell) \de_{\ell < g-4\ga}+ \sum_{m=0}^{\ga-1} (\ga - m) \de_{g-4\ga+2m \leq \ell \leq g-4\ga+2m+1}.
\]

If $k_j$ is divisible by 4, the linear part of the condition indexed by $q \in \mb{N}_{>k_j} \setminus {\rm S}_{g,\ga}$ is $a_{j,q}-\frac{k_j}{4} a_{1,q+4-k_j}$ as in Subcase 1.1 above; otherwise, the linear part of the condition indexed by $q \in \mb{N}_{>k_j} \setminus {\rm S}_{g,\ga}$ is $a_{j,q}- \frac{k_j-(4\ga+2)}{4} a_{1,q+4-k_j}- a_{j^{\ast},4\ga+2+q-k_j}$. The aggregate contribution $\rho(k_j)+k_j-j$ of $s=k_j$ to $r_P+b_P$ is

\[
(g+\ga+\ell-j) \de_{\ell < g-4\ga}+ \sum_{m=0}^{\ga-1} (5\ga +2\ell-m-j) \de_{g-4\ga+2m \leq \ell \leq g-4\ga+2m+1} \geq g
\]

in which equality holds if and only if $j=\ga+\ell-m^{\ast}$, where $m^{\ast}$ is either the unique integer for which $\ell \in [g-4\ga+2 m^{\ast},g-4\ga+2 m^{\ast}+1]$, or else $m^{\ast}=0$.

\medskip
However, just as in Subcase 1.1, the preceding argument needs to be adjusted because $2g-4\ga+1$ is not realizable as a positive sum of $k_1,\dots,k_n$, so the iterative procedure by which we walk up the column indexed by $k_j$ needs to be adjusted to explain those conditions induced by gaps greater than $2g-4\ga+1$. We will implement this adjustment in a unified way across all subcases following our preliminary analysis of Subcase 1.3.

\medskip
{\bf Subcase 1.3: $s=8\ga+4$.} In this subcase, $s$ admits either 3 or 2 distinct irreducible decompositions with respect to $k_1, \dots, k_n$, depending upon whether or not $s$ belongs to $\{k_{j^{\ast}+1},\dots,k_n\}$. The underlying partitions are $(4^{2\ga+1})$, $(4\ga+2)^2$, and possibly $(8\ga+4)$, if $k_j=8\ga+4$ for some $j$. Correspondingly we set $G:= f_{j^{\ast}}^2-f_1^{2\ga+1}$; and if $k_j=8\ga+4$, we further set $F_j:=f_j-f_1^{2 \ga+1}$. As before, we inductively walk up the column of the Dyck diagram indexed by $s$, perturbing $G$ and $F_j$ by monomials in $f_1,\dots,f_n$ at each step. The valuations of the resulting polynomials continue inscreasing until they reach $2g-4\ga$, at which stage no further iteration is possible, as no monomial in $f_1,\dots,f_n$ has valuation equal to $2g-4\ga+1$. Nevertheless, as a heuristic it is useful to provisionally ignore this obstruction and correct for the overcounting afterwards; in this idealization, each of the (inductively perturbed versions of) $G$ and $F_j$ (when $s=k_j$) would contribute $\rho(8\ga+4)$ algebraically independent conditions, with linear parts of the form $2a_{j^{\ast},q-4\ga-2}- (2\ga+1)a_{1,q-8\ga}$ and $a_{j,q}-(2\ga+1)a_{1,q-8\ga}$ for all $q \in \mb{N}_{>8\ga+4} \setminus {\rm S}_{g,\ga}$, respectively. The value of $\rho(8\ga+4)$ depends on how large $g$ is relative to $\ga$; namely,
\[
\rho(8\ga+4)= \de_{g \geq 6\ga+2}(g-5\ga-2)+ \sum_{j=0}^{\ga-1} (g-5\ga-1+j) \de_{6\ga-2j-1 < g \leq 6\ga-2j+1}.
\]
In this idealization, when $s \notin \{k_{j^{\ast}+1},\dots,k_n\}$, $s$ contributes $g-5\ga-1+j^{\ast \ast}$ to $r_P+b_P$. When $s=k_j$ for some $j$, $s$ contributes 
\[
2(g-5\ga-1+j^{\ast \ast})+ (8\ga+4-j)=2g-2\ga+2+2j^{\ast \ast}-j \geq 2g-5\ga-1+j^{\ast \ast}
\]
to $r_P+b_P$, in which the inequality is equality if and only if $j=3\ga+3+j^{\ast \ast}$.

\medskip
{\bf Conditions beyond $2g-4\ga+1$.} Because the minimal generator $2g-4\ga+1$ of ${\rm S}_{g,\ga}$ does not belong to ${\bf k}$, no monomial in the parameterizing functions $f_i,i=1,\dots,n$ has $t$-valuation $2g-4\ga+1$. In order to continue ascending the Dyck column indexed by a given element $s \in {\rm S}_{g,\ga}$ (and a pair of irreducible decompositions of $s$, that we fix at the outset) ``beyond" $2g-4\ga+1$, we perturb the inductively-constructed polynomial $F^s$ of valuation $2g-4\ga$ by (any) one of the inductively-constructed polynomials $F^{s^{\pr}}$ of valuation $2g-4\ga$ from a column labeled by a distinct element $s^{\pr} \in {\rm S}_{g,\ga}$. More precisely, we replace $F^s$ by $F^s-F^{s^{\pr}}$ and then continue our iterative process just as before, increasing the valuation of our polynomial by adding scalar multiples of monomials in $f_1,\dots,f_n$ at each step. When we do so for every column labeled by some $s \in {\rm S}_{g,\ga}$ that admits at least two irreducible decompositions, the net effect is that the lower bound on the codimension predicted by our naive idealization drops by $\rho(2g-4\ga+1)=\ga$; cf. Theorem~\ref{weak_bound} and its proof below.

\medskip
{\bf Minimizing the total number of conditions.} 

\medskip
{\bf Case 2:} {\it ${\bf k}$ contains odd entries.} Our analysis of conditions imposed by elements $s \in {\rm S}_{g,\ga}$ is identical to that in Case 1 whenever $s$ is even or strictly less than the minimal odd valuation $k_{\widehat{j}}$. Note that $k_{\widehat{j}} \geq 2g-4\ga+1$, and clearly $2g-4\ga+1>4\ga+2$ because $g \geq 4\ga+2$. The element $s=k_{\widehat{j}}$ contributes $k_{\widehat{j}}-\widehat{j}$ algebraically independent ramification conditions to $r_P+b_P$. Note that $k_{\widehat{j}}-\widehat{j} \geq g-\ga$, with equality if and only if $k_{\widehat{j}}=2g-4\ga+1$ and ${\bf k}$ includes all positive elements of ${\rm S}_{g,\ga}$ less than or equal to $2g-4\ga+1$. On the other hand, whenever $s>k_{\widehat{j}}$ and $s$ is odd, $s$ admits either 2 or 1 irreducible decompositions with respect to $k_1,\dots,k_n$, depending upon whether or not $s=k_j$ for some $j$. Once more, we may suppose without loss of generality that $s=k_j$; then $s$ contributes $\rho(k_j)+k_j-j$ algebraically independent conditions to $r_P+b_P$. By virtually the same argument as before, the total number of conditions arising from the parameterization is minimized when the valuation entries $k_j$ determine a consecutive sequence of elements in ${\rm S}_{g,\ga}$, with the caveat that the unique element $s=8\ga+4$ eligible to admit 3 irreducible decompositions might be skipped. Given an odd valuation $k_j= 2g-4\ga+1+ 4\ell$, where $\ell \in [0,\ga-1]$, we have $\rho(k_j)=\ga-\ell$ and therefore
\[
\rho(k_j)+k_j-j= 2g-3\ga+1+ 3\ell- j \geq g
\]
with equality if and only if $j=g-3\ga+1+3\ell$, which means precisely that ${\bf k}$ includes all positive elements of ${\rm S}_{g,\ga}$ less than or equal to $2g-4\ga+1+4\ell$.

\medskip
{\bf Aggregating codimension-minimizing conditions.} When $n \leq 3\ga+1$, every entry $k_j$ of ${\bf k}$ is strictly smaller than $8\ga+4$. Accordingly, we see that ${\rm cod}(\mathcal{V}_{\sss_{g,\gamma}},M^n_d)$ is at least
{\small
\[
\begin{split}
(n-2)g+3+ (4\ga+2-n)+ (g-5\ga-1+j^{\ast \ast}) -\ga-1 = (n-1)g- (2\ga+n-3-j^{\ast \ast}) \\
\text{ if } n \leq \ga; \text{ and }\\
(n-2)g+3+ (4\ga+2-(\ga+1))+ (g-5\ga-1+j^{\ast \ast}) -\ga-1 = (n-1)g- (3\ga-2-j^{\ast \ast}) \\
\text{ if } \ga+1 \leq n \leq 3\ga+1.
\end{split}
\]
}

\noindent But, if $n \geq 3\ga+2$, ${\rm cod}(\mathcal{V}_{\sss_{g,\gamma}},M^n_d)$ is bounded below 
{\small
\[
\begin{split}
(n-3)g+3+ (4\ga+2-(\ga+1))+ (2g-5\ga-1+j^{\ast \ast})-\ga-1 = (n-1)g-(3\ga-2-j^{\ast \ast}), \\
\text{ if } k_j= 8\ga+4 \text{ for some }j.
\end{split}
\]
}

\noindent We conclude that, whenever $n \geq \ga +1$,   
\[{\rm cod}(\mathcal{V}_{\sss_{g,\gamma}},M^n_d) \geq (n-1)g-(3\ga-2-j^{\ast \ast}). \].


\end{proof}

\begin{ex}\label{supporting_example}
\emph{Let $n=4$, $g=11$, and $\ga=2$, so that ${\rm S}=\langle 4,10,15 \rangle$ is the corresponding $\ga$-hyperelliptic semigroup of maximal weight; we will show that when ${\bf k}$ is chosen in such a way to minimize the number of conditions in Theorem~\ref{thmgam}, these conditions are in fact exhaustive. 
To this end, let $f=(f_2,f_4,f_5,f_6)$ denote a general element of $\mc{V}_{(4,8,10,12)}$, where $[t^i]f_j=a_{j,i}$; and let $F$ denote an arbitrary element of 
$\C[f_1,f_2,f_3,f_4]$. 
Without loss of generality we may assume, as in Example~\ref{exept1}, that $v_t(F)\geq 4$. We then have}
{\small
\begin{align*}
F&=
\alpha_{(4)} f_4+\al_{(2,2)} f_2^2 \\
&\ +\alpha_{(5)} f_5 \\
&\ +\alpha_{(6)} f_6+\al_{(4,2)} f_2 f_4 +\al_{(2^3)} f_2^3 \\
&\ +\al_{(5,2)} f_2 f_5 \\
&\ +\al_{(2^4)} f_2^4 + \al_{(4,2^2)} f_2^2 f_4+\al_{(4^2)} f_4^2 +\al_{(6,2)} f_2 f_6 \\
&\ + \al_{(5,2^2)} f_2^2 f_5 + \al_{(5,4)} f_4 f_5 \\
&\ +\al_{(2^5)} f_2^5 + \al_{(4,2^3)} f_2^3 f_4+\al_{(6,4)} f_4 f_6 + \al_{(5^2)} f_5^2 +\al_{(6,2^2)} f_2^2 f_6 +\al_{(4^2,2)} f_2 f_4^2
\end{align*}
}

\noindent \emph{in which the $\alpha_{\la}$ are complex coefficients and the $i$th line in the above diagram corresponds to the homogeneous component of $F$ of fixed valuation $2i<2g$.}

\medskip
\noindent \emph{Expanding $F$ as a power series in $t$, grouping terms of fixed valuation together, 
and implementing the method already applied in Example ~\ref{exept1} and Lemma~\ref{lemleq},}
\emph{we obtain}
$$
F=\psi_{s-1} t^{4s-3}+\phi_{s} t^{4s-2}+\psi_{s}t^{4s-1}+\phi_{s+1} t^{4s}+\psi_{s+1}t^{4s+1}+O(t^{4s+2})
$$ 
\emph{at step $s>3$, and set $\phi_{s}=\psi_{s}=\phi_{s+1}=0$, which in turn allows us to rewrite $\psi_{s+1}$ as a linear combination of polynomials in the coefficients of the parameterizing functions $f_i$.}

\medskip
\noindent \emph{By requiring $F$ to vanish to higher and higher order and recording the polynomials $G_i$ obtained at each step, we build the following table, in which the multipliers $c_1$, $c_2$ and $c_3$ are polynomials in the $a_{i,j}$, and {\it $G$- and $H$-polynomials} reference the decomposition $F=G+H$, where
\[
G=\alpha_{(4)} f_4+\al_{(2^2)} f_2^2+\alpha_{(5)} f_5+\alpha_{(6)} f_6+\al_{(2^3)} f_2^3+\al_{(5,2)} f_2 f_5+\al_{(2^4)} f_2^4+\al_{(2^5)} f_2^5+\al_{(5^2)} f_5^2.
\]
}

\begin{table}[h]\label{table:H-polynomials}
\centering
\begin{tabular}{c|l|l|l|l|l|l|l}
 & 
 \multicolumn{3}{c}{G-polynomials} & 
\multicolumn{4}{|c}{H-polynomials} \\ 
\hline                               
Gap & $\alpha_{(4)}$ & $\alpha_{(6)}$ & $\alpha_{(5^2)}$ & $\alpha_{(4,2)}$ & $\alpha_{(6,2)}$ & $\alpha_{(4^2)}$ & $\alpha_{(5,4)}$  \\
\hline    
9 & $G_1$ & & & & & &  \\
11 & $G_2$ & & & & & &   \\
13 & $G_3$ & $G_4$ &  & $G_1$ &  & &  \\
17 & $G_5$ &    &   & $G_3$ & $G_4$  & &  \\
21 & $G_6+c_1 G_7$ &   & $G_7$ & $G_5$ &   & $c_2 G_7$ & $c_3 G_7$  
\end{tabular}
\vspace{5pt}
\caption{$G$- and $H$-polynimials for $\ga$-hyperelliptic cusps}\label{tab2}
\end{table}

\emph{
}

\noindent \emph{Much as in the hyperelliptic case, the algebra generated by the $G$-polynomials in this example is precisely that generated by the polynomials $F$ and $F^{\ast}$ distinguished by the inductive process of Theorem~\ref{thmgam}; see Figure~\ref{Fig2} for a graphical representation. Inasmuch as every numerical semigroup ${\rm S}$ is $\ga$-hyperelliptic for some $\ga$, it seems natural to speculate that an analogous phenomenon persists for arbitrary ${\rm S}$.}

\end{ex}

\begin{figure}[H] 
\begin{center}
\includegraphics[width=5cm,height=5cm,keepaspectratio]{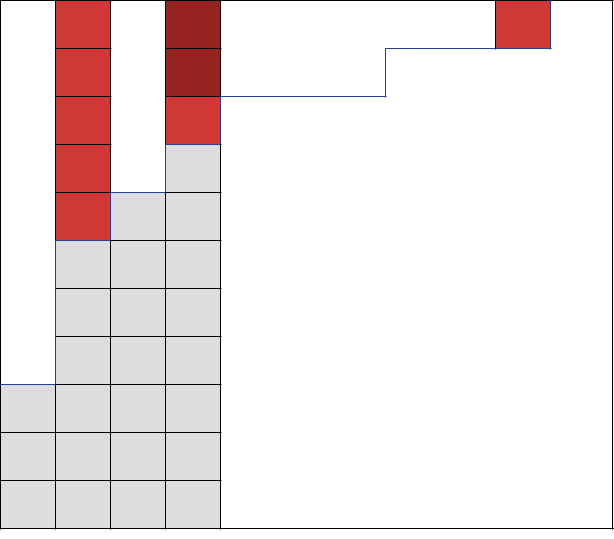}
\end{center}
\caption{Conditions contributing to $b_P$ and $r_P$ for ${\rm S}_{g,\ga}$ when $g=11$, $\ga=2$, and $n=4$. The dark red boxes do not contribute, i.e., they correspond to a correction to account for the fact that $2g-4\ga+1 \notin \rm{Span}({\bf k})$. In this case $\ga^{\ast \ast}=1$.} \label{Fig2}
\end{figure}

\medskip
In order to make this precise, we require one additional device, which will correct for possible ``syzygetic" redundancies among the polynomials $F$ and $F^{\ast}$ that will lead to over-counting otherwise. Namely, let $s_1<\dots<s_p$ denote the nonzero elements of ${\rm S}$ strictly less than the conductor, and let $\{(k_n^{m_{j,n}},\dots,k_1^{m_{j,1}}): j=1,\dots,\psi(s_i)\}$ denote the set of partitions underlying irreducible decompositions of $s_i$, $i=1,\dots,p$. For each $s_i$, let $v_{s_i,j}:=(m_{j,n},\dots,m_{j,1})$ denote the exponent vector of the $j$th indexing partition, $j=1,\dots,\psi(s_i)$. Let $V=V(E)$ denote the vector matroid on
\[
E:=\{v_{s_1,2}-v_{s_1,1},\ldots,v_{s_1,\psi(s_1)}-v_{s_1,1};\,\ldots\ \ldots\, ; v_{s_p,2}-v_{s_p,1},\ldots,v_{s_p,\psi(s_p)}-v_{s_p,1}\}.
\]

Denote the circuits of $V$ by $C_1,\ldots,C_q$ and for each $i=1,\dots,q$, let $s(i)$ be the largest semigroup element among $s_1,\dots,s_p$ for which $v_{s(i),j}-v_{s(i),1}\in C_i$ for some $j$. The {\it syzygetic defect} of $f$ with respect to ${\bf k}$ is
\[
D({\bf k}):= \sum_{i=1}^q\rho(s(i)).
\]
Finally, given $s \in {\rm S}$, let $\rho(s):=\#\{r>s: r\not\in {\rm S}\}$, let $\psi(s)$ denote the number of irreducible decompositions with respect to $(k_1,\dots,k_n)$, and let $\varphi(s):=\max(\psi(s)-1,0)$.
We always assume the ramification profile ${\bf k}=(0,k_1,\dots,k_n)$ in the cusp is fixed in advance.

\begin{conj}\label{gamma_hyp_conj}
Given a vector ${\bf k}=(0,k_1,\ldots,k_n)\in\mathbb{N}_{\geq 0}^{n+1}$, let $\mc{V}_{\bf k} \subset M^n_{d,g; {\rm S}}$ denote the subvariety parameterizing maps $f: \mb{P}^1 \ra \mb{P}^n$ with a unique cusp with semigroup ${\rm S}$ and ramification profile ${\bf k}$. Let $\{s_i^{\ast}\}_{i=1}^{\ell}$ denote the set of minimal generators of ${\rm S}$ strictly less than the conductor that do not appear as entries of ${\bf k}$; let $m= \min(\ell, \sum_{s \in {\rm S}} \varphi(s))$; and suppose $d=\deg(f) \geq \max(n,2g-2)$. Then
\begin{equation}\label{general_cod_estimate}
{\rm cod}(\mc{V}_{\bf k},M^n_d) =\sum_{i=1}^n(k_i-i)+\sum_{s\in {\rm S} \cap [2g]} \varphi(s)\rho(s)- \sum_{i=1}^{m} \rho(s^{\ast}_i)- D({\bf k})-1.
\end{equation}
\end{conj}

\begin{rem}\label{independent_partitions}
\emph{Certifying whether $\mathcal{V}_{\bf{k}}$ is nonempty in general is slightly delicate, inasmuch as it amounts to the assertion that the value semigroup of a {\it general} parameterization with ramification profile $\bf{k}$ contains every element of the underlying value semigroup ${\bf S}$.
On the other hand, whenever $\mathcal{V}_{\bf{k}}$ is nonempty, establishing that its codimension inside $M^n_d$ is {\it at least} the value predicted by the right-hand side of \eqref{general_cod_estimate} is relatively straightforward.}
\end{rem}

\begin{thm}\label{weak_bound}
With notation and hypotheses as in Conjecture~\ref{gamma_hyp_conj}, we have
\begin{equation}\label{general_weak_cod_estimate}
{\rm cod}(\mc{V}_{\bf k},M^n_d) \geq \sum_{i=1}^n(k_i-i)+\sum_{s\in {\rm S} \cap [2g]} \varphi(s)\rho(s) -\sum_{i=1}^{m} \rho(s^{\ast}_i)- D({\bf k})-1
\end{equation}
whenever $\mc{V}_{\bf k}$ is nonempty.
\end{thm}

\begin{proof}
The argument closely follows that used in proving Theorem~\ref{thmgam}. We start by treating the special case in which every minimal generator belongs to ${\bf k}$. Accordingly, let $f$ denote the ``universal" parameterization of $\mc{V}_{\bf k}$, represented by parameterizing functions $f_i=t^{k_i}+ a_{i,k_{i+1}}t^{k_i+1}+ O(t^{k_i+2})$, $i=1,\dots,n$. Given $s \in {\rm S}$, first assume that $s+1 \in \mb{N} \setminus {\rm S}$. Given any two partitions $(k_n^{m_n},\dots,k_1^{m_1})$ and $(k_n^{m_n^{\pr}},\dots,k_1^{m_1^{\pr}})$ underlying distinct irreducible decompositions of $s$, the binomial $F:=\prod_{i=1}^n f_i^{m_i}- \prod_{i=1} f_i^{m_i^{\pr}}$ imposes a nontrivial condition on the coefficients of $f$, namely that
\begin{equation}\label{pairwise_relation}
\text{lc}(F)= \sum_{i}(m_i-m_i^{\pr}) a_{i,k_{i+1}}=0
\end{equation}
where $\text{lc}(F)$ denotes the leading coefficient of $F$. Similarly, if $s+1 \in {\rm S}$, choose an irreducible decomposition $(k_n^{m_n^{\pr \pr}},\dots,k_1^{m_1^{\pr \pr}})$ of $s+1$, and set
\[
F^{(1)}:=F- \text{lc}(F) \prod_{i=1}^n f_i^{m_i^{\pr \pr}}= F-\sum_{i}(m_i-m_i^{\pr}) a_{i,k_{i+1}} \prod_{i=1}^n f_i^{m_i^{\pr \pr}}.
\]
The crucial point is that
\[
\text{lc}(F^{(1)})= \sum_{i}(m_i-m_i^{\pr}) a_{i,k_{i+2}}+ Q_1
\]
where the quadratic term $Q_1$ is nonlinear in the parameterizing coefficients $a_{i,j}$, and will be irrelevant for our purposes. Indeed, if $s+2 \in \mb{N} \setminus {\rm S}$, we obtain a condition, namely $\text{lc}(F^{(1)})=0$, whose {\it linear part} closely resembles \eqref{pairwise_relation} and involves as-yet unseen variables, namely $a_{i,k_{i+2}}$; if not, we choose an irreducible decomposition $(k_n^{m_n^{\pr \pr \pr}},\dots,k_1^{m_1^{\pr \pr \pr}})$ of $s+2$, and set $F^{(2)}:=F^{(1)}- \text{lc}(F^{(1)}) \prod_{i=1}^n f_i^{m_i^{\pr \pr \pr}}$.
Continuing in this way, we inductively walk up the column indexed by $s$ in the Dyck diagram associated with the pair $({\rm S},f)$, in the process recording the linear parts of the conditions associated with each of the elements of $\mb{N} \setminus {\rm S}$ strictly larger than $s$. 

\medskip

Note that these linear parts depend only on the pair of irreducible decompositions of $s$ we singled out at the outset. Moreover, if they are linearly independent, then their associated nonlinear conditions are algebraically independent. To push this logic further, let $\{(k_n^{m_{i,n}},\dots,k_1^{m_{i,1}}): i=1,\dots,\psi(s)\}$ denote the set of partitions underlying irreducible decompositions of $s$, and let
$v_{s,j}:=(m_{j,1},\dots,m_{j,n})$ denote the $j$th exponent vector, $j=1,\dots,\psi(s)$. We now fix a choice of reference exponent vector $v_{s,1}$. Each difference $v_{s,j}-v_{s,1}$, $j \neq 1$ indexes an initial pair of irreducible decompositions of $s$, and is associated with $\rho(s)$ conditions encountered while inductively walking up the column of the Dyck diagram indexed by $s$.  Let $V(s)$ denote the span of the vectors $v_{s,j}-v_{s,1}$, $2 \leq j \leq \psi(s)$.

\medskip
Now let $s_1 < \dots < s_p$ denote the nonzero elements of ${\rm S}$ strictly less than the conductor, and let $V_{\rm S}:= \sum_{i=1}^p V(s_i)$. The set of circuits of the vector matroid $V$ referenced in the statement of Conjecture~\ref{gamma_hyp_conj} indexes a minimal set of linear dependencies among elements of $V_{\rm S}$. As a result, the output of our inductive procedure includes $\rho(s(i))$ redundant linear expressions in the parameterizing coefficients associated with the maximal semigroup element $s(i)$ implicated in a given circuit $C_i$; and the total number of linearly independent expressions is precisely $\sum_{s\in {\rm S}} \varphi(s)\rho(s)- D({\bf k})$.

\medskip
To modify the above argument in the presence of minimal generators $s_i^{\ast} \notin {\bf k}$, we proceed as follows. Fix a nonzero element $s \in {\rm S}$ for which $\varphi(s)>0$. Fix the reference irreducible decomposition indexed by $v_{s,1}$ as above; let $s^{(1)}$ denote the largest element in ${\rm S}$ that belongs to $\text{Span}({\bf k})$ yet is less than $s_1^{\ast}$; and let $F^{s^{(1)}}_{v_{s,j}}$, $j=2, \dots, \psi(s)$ denote the polynomial with $t$-valuation $s^{(1)}$ inductively constructed in the inductive ``column-walking" process associated with the decomposition of $s$ labeled by $v_{s,j}$. If $\varphi(s)=1$, there is a single inductive process, labeled by $v_{s,2}$, and it terminates. Otherwise, for $j=3, \dots, \psi(s)$ we set $G_{v_{s,j}}:= F^{s^{(1)}}_{v_{s,j}}- F^{s^{(1)}}_{v_{s,2}}$. We are now left with $\varphi(s)-1$ inductive column-walking processes operative in column $s$, whose associated linear conditions are the linear parts of the coefficients of $G_{v_{s,j}}$, $j=3, \dots, \psi(s)$ of terms with $t$-valuation greater than $s^{(1)}$. For each of these processes, we continue ascending the column indexed by $s$, perturbing by monomials in $f_1,\dots,f_n$ at each step until either the column is exhausted, or else we produce polynomials $G^{s^{(2)}}_{v_{s,j}}$ with $t$-valuation equal to the largest element $s^{(2)} \in {\rm S}$ that belongs to $\text{Span}({\bf k} \sqcup s_1^{\ast})$ yet is less than $s_2^{\ast}$. If $\varphi(s)=2$, our unique inductive process terminates. Otherwise, for $j=4, \dots, \psi(s)$ we set $H_{v_{s,j}}:= G_{v_{s,j}}- F^{s^{(1)}}_{v_{s,3}}$ and continue proceeding upwards. We conclude by induction on the number of minimal generators of ${\rm S}$ not in ${\bf k}$.
\end{proof}

\begin{ex}
\emph{An instructive case is that of ${\rm S}=\langle 8,10,12; 25,29\rangle$ and ${\bf k}=(8,10,12)$. There are precisely two nonzero elements of ${\rm S} \cap [2g]$ for which $\varphi>0$ and $\rho>0$, namely $s=20$ and $s=24$. Theorem~\ref{weak_bound} predicts that the column indexed $s=20$ (resp., $n=24$) contributes $4-2=2$ (resp., $2-1=1$) conditions beyond ramification, where the corrections arise from the minimal generators 25 and 29 that do not belong to ${\bf k}$. Macaulay2 \cite{ArqCont} confirms that Conjecture~\ref{gamma_hyp_conj} holds in this case, i.e., that the algebraic conditions produced by our iterative procedure and enumerated by Theorem~\ref{weak_bound} are exhaustive.}
\end{ex}

\begin{rem}
\emph{In every case that we have computed, the exponent vectors $v_{s,j}$, $j=1,\dots,s$ are linearly independent, and thus $\dim V(s)=\varphi(s)$. It seems likely this is a general feature of sets of irreducible partitions with fixed size. In our context, it implies that syzygetic dependencies only occur among conditions associated with distinct columns in the Dyck diagram.}
\end{rem}

\subsection{Rational curves with $\ga$-hyperelliptic singularities of minimal weight}\label{minimal_weight_semigroups}
One immediate upshot of Theorem~\ref{thmgam} is that the codimension of a Severi variety $M^n_{d,g;{\rm S}}$ is never {\it unexpectedly small}, i.e., strictly less than $(n-2)g$, whenever ${\rm S}$ is equal to the $\ga$-hyperelliptic semigroup ${\rm S}_{g,\ga}$ of maximal weight. On the other hand, in \cite[Thm 2.3]{CFM1} we produced a particular infinite class of mapping spaces $M^n_{d,g;{\rm S}}$ of unexpectedly small codimension; the associated semigroups ${\rm S}$ are $\ga$-hyperelliptic semigroups of {\it minimal} weight, and the projective targets of the underlying parameterizations are of dimension $n \geq 8$.

\medskip
More precisely, say that ${\rm S}$ is a $\ga$-hyperelliptic semigroup of minimal weight, for some $\ga \geq 0$. This means precisely that the associated Dyck path $\mc{P}$ is a staircase with steps of unit height and width, or equivalently, that ${\rm S}={\rm S}^{\ast}_{g,\ga}$, where
\[
{\rm S}^{\ast}_{g,\ga}:=\langle 2\ga+2, 2\ga+4, \dots, 2g-2\ga,2g-2\ga+1,2g-2\ga+2,\dots \rangle.
\]
Our result \cite[Thm 2.3]{CFM1} establishes that when $g=3\ga+4$ and $n=\ga+1$, 
the minimal-ramification stratum $\mc{V}_{(2\ga+2,2\ga+4,\dots,2\ga+2n)} \sub M^n_{d,g;{\rm S}^{\ast}_{g,\ga}}$ is unexpectedly large whenever $n \geq 8$. It is natural to ask whether fixing ${\rm S}={\rm S}^{\ast}_{g,\ga}$ while varying the genus $g$ and target dimension $n$ leads to other excess examples $\mc{V}_{(2\ga+2,2\ga+4,\dots,2\ga+2n)}$ when $n$ belongs to the critical interval $[3,7]$. 

\medskip
On the other hand, if $n \leq \ga$, then $\mc{V}_{(2\ga+2,2\ga+4,\dots,2\ga+2n)}$ is in fact empty; indeed, no value $s\in[2\ga+2n+2,4\ga+2]$ can be realized by a polynomial in the $f_i$, $i=1,\dots,n$ if these have valuation vector $(2\ga+2,2\ga+4,\dots,2\ga+2n)$. Our final result handles the remaining cases, in which $n \in [3,7]$ and $\ga \leq n-1$.

\begin{prop}\label{sporadic_minimal_weight_cases}
Assume that $n \leq 2g$, $d \geq \max(2g-2,n)$, and $3 \leq n \leq 7$.
Then either $\mc{V}_{(2\ga+2,2\ga+4,\dots,2\ga+2n)}$ is empty, or else
\[
\text{\rm cod}(\mc{V}_{(2\ga+2,2\ga+4,\dots,2\ga+2n)}, M^n_d) \geq (n-2)g
\]
with the following twenty-one exceptions:
\begin{itemize}
    \item $n \in \{6,7\}$, $\ga=5$, and $g \in [21,24]$; or
    \item $n=7$, $\ga=6$, and $g \in [23,35]$.
\end{itemize}
Of these exceptions, thirteen certifiably underlie Severi varieties with excess components\footnote{As we explain below, we can check explicitly with Macaulay2 that the conditions furnished by Theorem~\ref{weak_bound} are in fact exhaustive in all of the exceptional cases for which either $\ga=5$ or $\ga=6$ and $g \leq 27$.}.
\end{prop}

\begin{proof}
In light of Theorem~\ref{unibranch_theorem} and the discussion above, we may assume $\ga \in [1,n-1]$. We have
\begin{equation}\label{ramif_estimate}
r_P = \sum_{j=2\ga+1}^{2\ga+n}j =2n\ga +\binom{n+1}{2}
\end{equation}
while 
\begin{equation}\label{rho_estimate}
\rho(2\ga+2k)=g-2\ga-k
\end{equation}
for all $k=1, \dots, g-2\ga$. Further note that the set of minimal generators less than the conductor $2g-2\ga$ and not belonging to the ramification profile is empty; $\rho(s)=0$ whenever $s \geq 2g-2\ga$; and for every $k=1,\dots,g-2\ga-1$  we have 
\[
\psi(2\ga+2k)=\psi_{n,\ga}(\ga+k;\ga+1,\ga+2,\dots,\ga+n)
\]
where $\psi_{n,\ga}(t;\ga+1,\ga+2,\dots,\ga+n)$ denotes the number of irreducible decompositions of $t$ with respect to $\ga+1,\ga+2,\dots,\ga+n$. 
Now let $\wt{\psi}_{n,\ga}(t):=\max(\psi(t;\ga+1,\ga+2,\dots,\ga+n)-1,0)$.
Applying Theorem~\ref{weak_bound} in tandem with \eqref{ramif_estimate} and \eqref{rho_estimate}, we are reduced to showing that
\begin{equation}\label{reduction_ineq}
(n-2)g \leq 2n\ga+ \binom{n+1}{2}-1+ \sum_{k=1}^{g-2\ga-1} (g-2\ga-k) \wt{\psi}_{n,\ga}(\ga+k)- D({\bf k}).
\end{equation}

Our basic strategy, outside of the twenty-one exceptional cases, will be to find $j$ values of these $t$ for which the associated values of $\wt{\psi}_{n,\ga}(t)$ sum to at least $(n-2+j)$ (and higher in cases with nonzero syzygetic defect).

{\fl \it Case: $n=3$.} In light of \eqref{ramif_estimate}, the estimate \eqref{reduction_ineq} follows trivially whenever $g<6\ga+6$; so without loss of generality, we assume $g \geq 6\ga+6$.
Note that $\wt{\psi}_{3,\ga}(2\ga+4) \geq 1$; 
indeed, $2\ga+4$ has irreducible decompositions with underlying partitions $(\ga+3,\ga+1)$ and $(\ga+2)^2$. The required estimate \eqref{reduction_ineq} follows immediately.

{\fl \it Case: $n=4$.} This time, we may assume $g \geq 4\ga+5$ without loss of generality. The estimate \eqref{reduction_ineq} follows from the facts that
$\wt{\psi}_{4,\ga}(2\ga+4) \geq 1$, $\wt{\psi}_{5,\ga}(2\ga+5) \geq 1$, and that there are no syzygetic dependencies among the irreducible decompositions of $2\ga+4$ and $2\ga+5$ with underlying partitions $(\ga+2)^2$, $(\ga+3,\ga+1)$ and $(\ga+3,\ga+2)$, $(\ga+4,\ga+1)$, respectively.

{\fl \it Case: $n=5$.} We may assume $g \geq \frac{10}{3}\ga+\frac{14}{3}$ without loss of generality. As a result, the right-hand side of \eqref{reduction_ineq} is at least
\[
\ka_5(g,\ga):=10\ga+14+ \sum_{k=1}^{\lceil \frac{4}{3}\ga+\frac{11}{3} \rceil} (g-2\ga-k) \wt{\psi}_{5,\ga}(\ga+k)- D({\bf k}).
\]
Obtaining the required $3g$ conditions now depends on the value of $\ga$ itself. 
\begin{itemize}
\item When $\ga=1$, we compute $\ka_5(g,1) \geq 3g+6$ using the facts that $\wt{\psi}_{5,1}(k+1) \geq 1$ when $k=3,4,5$, and that there are no syzygetic dependencies among the irreducible decompositions of 4,5, and 6 with underlying partitions $(2^2)$, $(4)$; $(3,2)$, $(5)$; and $(2^3)$, $(6)$, respectively.

\item When $\ga=2$, we compute $\ka_5(g,2)= 3g+7$ using the facts that $\wt{\psi}_{5,2}(k+2)=1$ when $k=4,5,6$, and that there are no syzygetic dependencies among the irreducible decompositions of 6,7, and 8 with underlying partitions $(3^2), (6)$; $(4,3), (7)$; and $(4^2),(5,3)$, respectively.

\item When $\ga=3$, we compute $\ka_5(g,3)= 3g+6$ using the facts that $\wt{\psi}_{5,3}(k+3)=1$ when $k=5,7,8$, and that there are no syzygetic dependencies among the irreducible decompositions of 8,10, and 11 with underlying partitions $(4^2), (8)$; $(5^2), (6,4)$; and $(6,5),(7,4)$, respectively.

\item When $\ga=4$, we compute $\ka_5(g,4)=2g+21$ using the facts that $\wt{\psi}_{5,4}(k+4)=1$ when $k=8,9$; in particular, we obtain the required estimate whenever $g \leq 21$. Now say $g \geq 22$. Note that $\wt{\psi}_{5,4}(14) \geq 1$, and that there are no syzygetic dependencies among the irreducible decompositions of 12, 13, and 14 with underlying partitions $(6^2), (7,5)$; $(8,5), (7,6)$; and $(9,5), (8,6)$, respectively. It follows 
that the right-hand side of \eqref{reduction_ineq} is at least
$\ka(g,4)+ (g-2\ga-10)=3g+3$.
\end{itemize}

{\fl \it Case: $n=6$.} We may assume $g \geq 3\ga+5$; the right-hand side of \eqref{reduction_ineq} is then at least
\[
\ka_6(g,\ga):=12\ga+20+ \sum_{k=1}^{\ga+4} (g-2\ga-k) \wt{\psi}_{6,\ga}(\ga+k) -D({\bf k}).
\]
\begin{itemize}
    \item When $\ga=1$, we compute $\ka_6(g,1)=4g+7$ using the facts that $\wt{\psi}_{6,1}(k+1)=1$ for $k=3,4$, $\wt{\psi}_{6,1}(6)=2$, and that the corresponding collection of irreducible decompositions (whose underlying partitions are $(2^2), (4)$; $(3,2), (5)$; and $(4,2), (3^2), (6)$) has zero syzygetic defect.
    \item When $\ga=2$, we compute $\ka_6(g,2)=4g+7$ using the facts that $\wt{\psi}_{6,2}(k+2)=1$ for $k=4,5$, $\wt{\psi}_{6,2}(8)=2$, and that the corresponding collection of irreducible decompositions (whose underlying partitions are $(3^2), (6)$; $(4,3), (7)$; and $(4^2), (5,3), (8)$) has zero syzygetic defect.
    \item When $\ga=3$, we compute $\ka_6(g,3)=3g+20$ using the facts that $\wt{\psi}_{6,3}(k+3)=1$ for $k=5,6,7$, and that the corresponding collection of irreducible decompositions (with underlying partitions $(4^2), (8)$; $(5,4), (9)$; and $(5^2), (6,4)$) has zero syzygetic defect. Thus we have produced at least $4g$ conditions whenever $g \leq 20$. Assume $g \geq 21$; we deduce that the right-hand side of \eqref{reduction_ineq} is at least $4g+6$ from the facts that $\wt{\psi}_{6,3}(11)=1$, and that the corresponding irreducible decompositions with underlying partitions $(7,4), (6,5)$ are independent of the others already listed.
    \item When $\ga=4$, we use the fact that $\wt{\psi}_{6,4}(k+4)=1$ for $k=6,8$ and that the corresponding irreducible decompositions are independent, i.e. have zero syzygetic defect, to compute $\ka_6(g,4)=2g+38$, which is strictly less than $4g$ if and only if $g \geq 20$. Assume $g \geq 20$; using the facts that $\wt{\psi}_{6,4}(13)=1$ and $\wt{\psi}_{6,4}(14)=2$ and that the collection of irreducible decompositions with underlying partitions $(5^2),(10)$; $(6^2),(7,5)$; $(7,6), (8,5)$; and $(7^2),(9,5),(8,6)$ has zero syzygetic defect, we deduce that the right-hand side of \eqref{reduction_ineq} is at least $\ka_6(g,4)+2g-35=4g+3$.
    \item When $\ga=5$, we use the fact that $\wt{\psi}_{6,5}(9)=1$ to compute $\ka_6(g,5)=g+61$, which is strictly less than $4g$ if and only if $g \geq 21$. Now assume $g \geq 25$; we will turn to the (exceptional) cases $g=21,22,23,24$ in a moment. To bound the right-hand side of \eqref{reduction_ineq}, we use the facts that $\wt{\psi}_{6,5}(15)=1$ counts (one minus) the irreducible decompositions with underlying partitions $(8,7)$, $(9,6)$; $\wt{\psi}_{6,5}(16)=2$ counts the irreducible decompositions corresponding to $(10,6), (9,7), (8^2)$; $\wt{\psi}_{6,5}(17)=2$ counts the irreducible decompositions corresponding to $(9,8), (10,7), (11,6)$; $\wt{\psi}_{6,5}(18)=3$ counts the irreducible decompositions corresponding to $(9^2), (10,8), (11,7), (6^3)$; and $\wt{\psi}_{6,5}(19)=2$ counts the irreducible decompositions corresponding to $(10,9),(11,8),(7,6^2)$. The new wrinkle here is that the syzygetic defect is not zero; rather, there are five nontrivial linear dependencies which together account for a correction of $5g- 113$. It follows that the right-hand side of \eqref{reduction_ineq} is at least $6g-49$, which is greater than $4g$ because $g \geq 25$.
    
    \noindent
    Finally, if $g=21$, the only irreducible decompositions that are operative are those corresponding to partitions of 15 as above; accordingly, the syzygetic defect is zero and the right-hand side of \eqref{reduction_ineq} becomes $2g+41=83$. Similarly, if $g=22$, $g=23$, or $g=24$, we also allow for decompositions of 16, 17, and 18, and the number of syzygetic linear dependencies is one, two, or four, respectively, so the right-hand side of \eqref{reduction_ineq} becomes $3g+20=86$, $4g-2=90$, or $5g-25=95$, respectively. Calculations with Macaulay2 \cite{ArqCont} certify that these are the actual codimensions of the corresponding loci $\mc{V}_{\bf k}$.
    
\end{itemize}

{\fl \it Case: $n=7$.} We may assume $g \geq \frac{14}{5}\ga+\frac{27}{5}$; the right-hand side of \eqref{reduction_ineq} is then at least
\[
\ka_7(g,\ga):=14\ga+ 27+ \sum_{k=1}^{\lceil \frac{4}{5}\ga+\frac{22}{5} \rceil} (g-2\ga-k) \wt{\psi}_{7,\ga}(\ga+k)- D({\bf k}).
\]

\begin{itemize}
    \item When $\ga=1$, we compute $\ka_7(g,1)= 4g+16$ using the facts that $\wt{\psi}_{7,1}(k+1)=1$ for $k=3,4,6$ and $\wt{\psi}_{7,1}(6)=2$, and that the corresponding irreducible decompositions with underlying partitions $(2^2),(4)$; $(3,2),(5)$; $(2^2,3), (7)$; and $(3^2),(2^3),(6)$ are independent. In particular, it follows that the right-hand side of \eqref{reduction_ineq} is at least $5g+8$.
    \item When $\ga=2$, we count much as in the $\ga=1$ case. It is easy to see that $\wt{\psi}_{7,2}(k+2)=1$ for $k=4,5$ and $\wt{\psi}_{7,2}(8)=2$, and that the corresponding collections of irreducible decompositions with underlying partitions $(3^2), (6)$; $(4,3), (7)$; and $(4^2), (5,3), (8)$ are independent. It follows that $\ka_7(g,2)=4g+18$, which is at least $5g$ when $g \leq 18$. On the other hand, if $g \geq 19$ we use $\wt{\psi}_{7,2}(7)=2$ and that the corresponding irreducible decompositions indexed by $(3^3), (5,4), (9)$ are independent of the others to conclude that the right-hand side of \eqref{reduction_ineq} is at least $5g+7$.
    \item When $\ga=3$, we compute $\ka_7(g,3)=4g+20$ using the facts that $\wt{\psi}_{7,3}(k+3)=1$ for $k=5,6$, $\wt{\psi}_{7,3}(10)=2$, and that the corresponding collections of irreducible decompositions indexed by $(4^2), (8)$; $(5,4), (9)$; and $(6,4), (5^2), (10)$ are independent. In particular, we obtain at least $5g$ conditions whenever $g \leq 20$. If $g \geq 21$, we use $\wt{\psi}_{7,3}(8)=1$ and the fact that the irreducible decompositions $(7,4), (6,5)$ are independent of the others already listed to conclude that the right-hand side of \eqref{reduction_ineq} is at least $5g+6$.
    \item When $\ga=4$, we compute $\ka_7(g,4)=3g+38$ using the facts that $\wt{\psi}_{7,4}(k+4)=1$ for $k=6,7,8$ and that the corresponding collections of irreducible decompositions indexed by $(5^2),(10)$; $(6,5),(11)$; and $(7,5),(6^2)$ are independent. Consequently, we obtain at least $5g$ conditions whenever $g \leq 19$. If $g \geq 20$, there are additional conditions arising from the irreducible decompositions indexed by $(7,6), (8,5)$; and $(7^2), (8,6), (9,5)$ and counted by $\wt{\psi}_{7,4}(13)=1$ and $\wt{\psi}_{7,4}(14)=2$, respectively. These are not independent of the others; rather, there is a single linear relation, which produces a syzygetic defect of $g-2\ga-10=g-18$. It follows that the right-hand side of \eqref{reduction_ineq} is at least $5g+3$.
    \item When $\ga=5$, we compute $\ka_7(g,5)=2g+61$ using the facts that $\wt{\psi}_{7,5}(k+5)=1$ for $k=7,9$ and that the corresponding collections of irreducible decompositions indexed by $(6^2),(12)$ and $(7^2),(8,6)$ are independent. Consequently, we obtain at least $5g$ conditions whenever $g \leq 20$. Now assume $g \geq 25$. To bound the right-hand side of \eqref{reduction_ineq}, we use the facts that $\wt{\psi}_{7,5}(15)=1$ counts irreducible decompositions indexed by $(9,6),(8,7)$; $\wt{\psi}_{7,5}(16)=2$ counts irreducible decompositions indexed by $(10,6), (9,7), (8^2)$; $\wt{\psi}_{7,5}(17)=2$ counts irreducible decompositions indexed by $(11,6),(10,7),(9,8)$; $\wt{\psi}_{7,5}(18)=3$ counts irreducible decompositions indexed by $(6^3),(11,7),(10,8),(9^2)$; and $\wt{\psi}_{7,5}(19)=2$ counts irreducible decompositions indexed by $(7,6^2),(11,8),(10,9)$. By the same calculation carried out previously when $n=6$ and $\ga=5$, the associated syzygetic defect is $5g- 113$. It follows that the right-hand side of \eqref{reduction_ineq} is at least $7g-49$, which is greater than $5g$ because $g \geq 25$.
    
    \noindent Similarly, if $g=21$, the right-hand side of \eqref{reduction_ineq} becomes $3g+41=104$; while if $g=22$, $g=23$, or $g=24$, it becomes $4g+20=108$, $5g-2=113$, or $6g-25=119$, respectively.
    \item Finally, when $\ga=6$, we use the fact that $\wt{\psi}_{7,6}(16)=1$ counts irreducible decompositions of 16 indexed by $(9,7),(8^2)$ to compute $\ka_7(g,6)=g+89$, which is at least $5g$ whenever $g \leq 22$. Now say $g \geq 29$. To bound the right-hand side of \eqref{reduction_ineq}, we use the facts that $\wt{\psi}_{7,6}(17)=1$ counts irreducible decompositions indexed by $(10,7)$, $(9,8)$; $\wt{\psi}_{7,6}(18)=2$ counts irreducible decompositions indexed by $(11,7)$, $(10,8)$, $(9^2)$; $\wt{\psi}_{7,6}(19)=2$ counts irreducible decompositions indexed by $(12,7)$, $(11,8)$, $(10,9)$; $\wt{\psi}_{7,6}(20)=3$ counts irreducible decompositions indexed by $(13,7)$, $(12,8)$, $(11,9)$, $(10^2)$; $\wt{\psi}_{7,6}(21)=3$ counts irreducible decompositions indexed by $(7^3)$, $(13,8)$, $(12,9)$, $(11,10)$; $\wt{\psi}_{7,6}(22)=3$ counts irreducible decompositions indexed by $(8,7^2)$, $(13,9)$, $(12,10)$, $(11^2)$; and $\wt{\psi}_{7,6}(23)=2$ counts irreducible decompositions indexed by $(7^2,9)$, $(8^2,7)$, $(13,10)$, $(12,11)$.\footnote{In the last case, the fact that $\wt{\psi}_{7,6}(23)$ is {\it two} less than the number of underlying partitions arises because the decomposition $8+8+7=7+7+9$ is {\it reducible}.} It is, moreover, easy to check that for every integer $k \in [11,15]$, there is exactly one irreducible decomposition of $k+6$ counted by $\wt{\psi}_{7,6}(k+6)$ that is independent of the irreducible decompositions counted by $\wt{\psi}_{7,6}(10)$, $\wt{\psi}_{7,6}(11)$, \dots, $\wt{\psi}_{7,6}(k+5)$. It follows that the right-hand side of \eqref{reduction_ineq} is at least 
    \[
    \ka_7(g,6)+ \sum_{k=11}^{15} (g-12-k)= 6g-36
    \]
    which is greater than $5g$ precisely when $g \geq 36$. Furthermore, it is easy to check explicitly that no irreducible decomposition of $k+6$ counted by $\wt{\psi}_{7,6}(k+6)$ is independent of the irreducible decompositions counted by $\wt{\psi}_{7,6}(10)$, $\wt{\psi}_{7,6}(11)$, \dots, $\wt{\psi}_{7,6}(k+5)$ for integers $k \in [16,22]$, so Conjecture~\ref{gamma_hyp_conj} predicts that the codimension of $\mc{V}_{(14,16,\dots,26)}$ is precisely $(6g-36)$ for all $g \in [29,35]$.

    \noindent Similarly, the right-hand side of \eqref{reduction_ineq} is precisely 112 when $g=23$; 114 when $g=24$; 117 when $g=25$; 121 when $g=26$; 126 when $g=27$; and 133 when $g=28$. Macaulay2 \cite{ArqCont} certifies the actual codimensions of the corresponding loci $\mc{V}_{\bf k}$ whenever $g \leq 27$; we anticipate this can be pushed further with a more sophisticated implementation.
\end{itemize}
\end{proof}

\begin{rem}

We suspect that Conjecture~\ref{gamma_hyp_conj} should persist for arbitrary choices of algebraically-closed base fields, in the same way that the irreducibility of the classical Severi variety of plane curves persists \cite{CHT}. If true, this should be easy to verify in each instance for which Conjecture~\ref{gamma_hyp_conj} may be checked by computer.
\end{rem}

To close this subsection, we show how to modify the construction of \cite[Thm 2.3]{CFM1} to obtain new infinite families of Severi varieties 
of unexpectedly small codimension, assuming the validity of Conjecture~\ref{gamma_hyp_conj}.

\begin{prop}\label{new_Severi-excessive_families} Let $\mc{V}_{(2\ga+2,2\ga+4,\dots,2\ga+2n)}$ denote the (generic stratum of a) Severi variety with underlying value semigroup ${\rm S}^{\ast}_{g,\ga}$ as above. Now assume that $n=\ga+1$, $d \geq 2g-2$, and that $g=3\ga+6$ (resp., $g=3\ga+8$) for some nonnegative integer $\ga \geq 5$. Assume that Conjecture~\ref{gamma_hyp_conj} holds; then $\mc{V}_{(2\ga+2,2\ga+4,\dots,2\ga+2n)}$ is of codimension $\frac{5}{2}\ga^2+ \frac{7}{2}\ga+3$ (resp., $\frac{5}{2}\ga^2+ \frac{7}{2}\ga+10$) in $M^n_d$. In particular, $\mc{V}_{(2\ga+2,2\ga+4,\dots,2\ga+2n)}$ is of codimension strictly less than $(n-2)g$ in $M^n_d$.
\end{prop}

\begin{proof}
Assuming that Conjecture~\ref{gamma_hyp_conj} holds, the right-hand side of \eqref{reduction_ineq} computes the codimension of $\mc{V}=\mc{V}_{(2\ga+2,2\ga+4,\dots,2\ga+2n)}$ in $M^n_d$.
{\fl \it Case: $g=3\ga+6$.} The right-hand side of \eqref{reduction_ineq} is equal to
\[
\frac{5}{2}\ga^2+ \frac{7}{2}\ga+ \sum_{k=1}^{\ga+5} (\ga+6-k) \wt{\psi}_{\ga+1,\ga}(\ga+k)- D({\bf k}).
\]
The desired conclusion in this case follows immediately from the facts that $\wt{\psi}_{\ga+1,\ga}(\ga+k)=0$ for $k \leq \ga+3$, while $\wt{\psi}_{\ga+1,\ga}(\ga+k)=1$ for $k=\ga+4,\ga+5$. The fact that $D({\bf k})=0$, in particular, is clear.

{\fl \it Case: $g=3\ga+8$.} This time, the right-hand side of \eqref{reduction_ineq} is equal to
\[
\frac{5}{2}\ga^2+ \frac{7}{2}\ga+ \sum_{k=1}^{\ga+7} (\ga+8-k) \wt{\psi}_{\ga+1,\ga}(\ga+k)- D({\bf k}).
\]
We use the facts that $\wt{\psi}_{\ga+1,\ga}(k)=0$ for $k \leq \ga+3$, while $\wt{\psi}_{\ga+1,\ga}(2\ga+4)=1$ counts the irreducible decompositions indexed by $((\ga+2)^2),((\ga+3),(\ga+1))$; $\wt{\psi}_{\ga+1,\ga}(2\ga+5)=1$ counts the irreducible decompositions indexed by $((\ga+3),(\ga+2)),((\ga+4),(\ga+1))$; $\wt{\psi}_{\ga+1,\ga}(2\ga+6)=2$ counts the irreducible decompositions indexed by $(\ga+5,\ga+1),(\ga+4,\ga+2),((\ga+3)^2)$; and $\wt{\psi}_{\ga+1,\ga}(2\ga+7)=2$ counts the irreducible decompositions indexed by $(\ga+6,\ga+1),(\ga+5,\ga+2),((\ga+3)^2)$. The two linear dependencies among these together account for $D({\bf k})=\sum_{k=\ga+6}^{\ga+7}(g-2\ga-k)=3$, and the desired conclusion follows.
\end{proof}
The upshot of Proposition~\ref{new_Severi-excessive_families}, taken together with \cite[Thm 2.3]{CFM1}, is that we expect unexpectedly large Severi varieties to exist in every genus $g \geq 21$ and every projective target dimension $n \geq 6$.

\subsection{Generic cusps with ramification $(2m, 2m+2,\dots,2m+2n-2)$, $n \geq 3$}\label{generic_semigroups}

In this subsection, we study the value semigroup ${\rm S}_{\rm gen}$ of a generic singularity with ramification profile equal to an $n$-tuple of consecutive even numbers. Whenever $n \geq 4$, we are not able to determine ${\rm S}_{\rm gen}$ explicitly; we still manage to show, however, that its structure forces the associated Severi variety $M^n_{d,g;{\rm S}_{\rm gen}}$ to be unexpectedly large in general. When $n=3$, we are able to determine ${\rm S}_{\rm gen}$ explicitly and show that $M^n_{d,g;{\rm S}_{\rm gen}}$ is usually unexpectedly large as a result.

\begin{thm}\label{generic_semigroups_n_geq_4}
Given positive integers $n \geq 4$ and $m \geq n+1$, let ${\rm S}_{\rm gen}={\rm S}_{\rm gen}(m,n)$ denote the value semigroup of a generic cusp in $\mb{C}^n$ with ramification profile $(2m,2m+2,\dots,2m+2n-2)$. We then have
\begin{small}
\[
\begin{split}
&\{1,\dots,2m-1\} \sqcup \{2m+1, 2m+3, \dots, 4m+3\} \sqcup \{2m+2n, 2m+2n+2, \dots, 4m-2\} \\
&\sqcup \{4m+4n-5, 4m+4n-3,\dots, 6m+3\} \sub {\rm G}_{\rm gen}
\end{split}
\]
\end{small}
where ${\rm G}_{\rm gen}= \mb{N} \setminus {\rm S}_{\rm gen}$ denotes the complement of ${\rm S}_{\rm gen}$.
\end{thm}

\begin{proof}
Fix a choice of generic $n$-tuple of power series $f(t)=(f_1(t), \dots, f_n(t))$ with ramification profile $(2m,2m+2,\dots,2m+2n-2)$ in $t=0$; without loss of generality, we may assume each of the parameterizing functions $f_i(t)$ has a monic lowest-order term. The generic value semigroup ${\rm S}_{\rm gen}$ comprises all valuations realized by polynomials in the $f_i(t)$, $i=1,\dots,n$. It follows immediately that $\{1,\dots,2m-1\} \sub {\rm G}_{\rm gen}$. Similarly, we have
\[
\{2m+1, 2m+3, \dots, 4m+3\} \sqcup \{2m+2n, 2m+2n+2, \dots, 4m-2\} \sub {\rm G}_{\rm gen}
\]
as the minimal valuation realized by a {\it nonlinear} polynomial in the $f_i(t)$ is $v_t(f_2^2-f_1f_3)= 4m+5$. (Note that the latter equality is a consequence of genericity.) To conclude, it suffices to see that no element belonging to $\{4m+4n-5, 4m+4n-3,\dots, 6m+3\}$ is the valuation of a quadratic polynomial in the $f_i(t)$; indeed, by genericity, the largest {\it odd} valuation realized by such a quadratic polynomial is $v_t(f_{n-1}^2-f_{n-2}f_n)=4m+4n-7$.
\end{proof}

For our purposes, the crucial take-away of Theorem~\ref{generic_semigroups_n_geq_4} is the following.
\begin{coro}\label{unexpectedly_large_examples_from_generic_semigroups}
For every pair of positive integers $n \geq 4$ and $m > \frac{7}{6}n+\frac{26}{9}$, the corresponding Severi variety $M^n_{d,g;{\rm S}_{\rm gen}}$ is of codimension strictly less than $(n-2)g$ in $M^n_d$ whenever $d \gg g$.
\end{coro}

\begin{proof}
By genericity, the only conditions imposed on holomorphic maps by ${\rm S}_{\rm gen}$ arise from ramification in the preimage $P$ of the cusp, and there are 
\[
r_P-1=\sum_{i=1}^n (2m+2i-2) -1= 2mn+ \binom{n-1}{2}-1
\]
of these. On the other hand, Theorem~\ref{generic_semigroups_n_geq_4} implies that
\[
g \geq 5m-3n+6
\]
where $g=\#{\rm G}_{\rm gen}$ is the delta-invariant of a generic cusp with ramification $(2m,2m+2,\dots,2m+2n-2)$. Accordingly, it suffices to check that
\begin{equation}\label{mn_inequality}
2mn+ \binom{n-1}{2}-1 < (n-2)(5m-3n+6)
\end{equation}
and this is ensured by our numerical hypotheses on $m$ and $n$.
\end{proof}

When $n=3$, we have to work harder to produce unexpectedly large Severi varieties; indeed, the inequality~\eqref{mn_inequality} {\it never} holds. To do so, we exploit a natural connection between generators of ${\rm S}_{\rm gen}$ and solutions to linear diophantine equations that is particular to the $n=3$ case.

\begin{thm}\label{generic_semigroup_n=3}
For every positive integer $m$, the value semigroup ${\rm S}_{\rm gen}$ of a generic parameterization with ramification profile $(2m,2m+2,2m+4)$ is equal to ${\rm S}^0=\langle 2m,2m+2,2m+4, 4m+5, (m+2)m+1 \rangle$ 
(resp., ${\rm S}^1=\langle 2m,2m+2,2m+4, 4m+5, (m+3)m+1 \rangle$) if $m$ is even (resp., odd). In particular, ${\rm S}_{\rm gen}$ is of genus $\frac{1}{2}m^2+ 2m-1$ (resp.,$\frac{1}{2}m^2+ 2m-\frac{3}{2}$) if $m$ is even (resp., odd).
\end{thm}

\begin{proof}
Let $f=(f_1,f_2,f_3)$ denote a generic power series map with ramification profile $(2m,2m+2,2m+4)$ in $t=0$; then $f_i=f_i(t)$ is a power series with $t$-adic valuation $v_i=2m+2i-2$, $i=1,2,3$. Now assume that $m$ is even. We begin by showing that ${\rm S}^0 \sub {\rm S}_{\rm gen}$. Without loss of generality, assume that each $f_i$ has a monic term of lowest order. Genericity then implies that $v_t(f_2^2-f_1f_3)= 4m+5$ belongs to ${\rm S}_{\rm gen}$. That is to say, the fact that $4m+5 \in {\rm S}_{\rm gen}$ arises because the diophantine equation $ma+ (m+2)b=(m+1)c$ admits a triple of nonnegative solutions $(a,b,c)=(1,1,2)$. Similarly, the diophantine equation $ma+ (m+1)b=(m+2)c$ admits nonnegative solutions $(a,b,c)=(\frac{m}{2}+1,0,\frac{m}{2})$, which implies that $v_t(f_1^{\frac{m}{2}+1}-f_3^{\frac{m}{2}})=m(m+2)+1$ belongs to ${\rm S}_{\rm gen}$. It follows that ${\rm S}^0 \sub {\rm S}_{\rm gen}$.

\medskip
Note that the gap set ${\rm G}^0= \mb{N} \setminus {\rm S}^0$ is given by
{\tiny
\[
\begin{split}
{\rm G}^0&=\{1,\dots,2m-1\} \bigsqcup \{2m+1, 2m+3,\dots,4m+3\} \bigsqcup \bigsqcup_{i=1}^{\frac{m}{2}-1} 2\{im+2i+1,\dots,(i+1)m-1\}\\
&\bigsqcup \bigsqcup_{i=1}^{\frac{m}{2}-2} \{2(i+1)m+2(2i+1)+1, 2(i+1)m+2(2i+1)+3, \dots, 2(i+2)m+3\} \bigsqcup \{m^2+ 2m-1\}.
\end{split}
\]
}
We now claim that no element $g \in {\rm G}^0$ is the valuation $v$ of a polynomial in the parameterizing functions $f_i$, $i=1,2,3$. 
Indeed, this is clear for every $g \in \{1,\dots,2m-1\}$, as every such $g$ is strictly less than every $v_i$; as well as for every $g \in \{2m+1, 2m+3,\dots,4m+3\}$, since by construction $4m+5$ is the minimally realizable odd valuation. Similarly, our construction shows that each of the remaining elements $g$, if equal to the valuation of a polynomial in the $f_i$, is necessarily the valuation of the sum of at most two {\it monomials} in the $f_i$ (since there are no ``triple ties" among valuations of monomials in this range). Moreover, since no $g$ belongs to $\langle 2m, 2m+2, 2m+4 \rangle$, this means that any realizable $g$ is necessarily the valuation of the sum of two monomials with equal valuation $v$. Genericity of the $f_i$ then forces $g$ to be $v+1$, where $v$ is a multiple of $4m+4$; as no $g$ fits this description, it follows that ${\rm S}_{\rm gen}={\rm S}^0$ when $m$ is even.
As a result, ${\rm S}_{\rm gen}$ has genus
\[
g(m)= (2m-1)+ (m+2)+ \sum_{i=1}^{\frac{m}{2}-1}(m-2i-1)+ \sum_{i=1}^{\frac{m}{2}-2} (m-2i+1)+ 1=\frac{1}{2}m^2+ 2m-1.
\]

\medskip
Assume now that $m$ is odd. Once again, $v_t(f_2^2-f_1f_3)= 4m+5$ belongs to ${\rm S}_{\rm gen}$, as does $v_t(f_1^{\frac{m+3}{2}}-f_2f_3^{\frac{m-1}{2}})=(m+3)m+1$; so ${\rm S}^1 \sub {\rm S}_{\rm gen}$. In this case the gap set ${\rm G}^1= \mb{N} \setminus {\rm S}^1$ is
{\tiny
\[
\begin{split}
    {\rm G}^1 &= \{1,\dots,2m-1\} \bigsqcup \{2m+1, 2m+3,\dots,4m+3\} \bigsqcup \bigsqcup_{i=1}^{\frac{m-1}{2}} 2\{im+2i+1,\dots,(i+1)m-1\}\\
    & \bigsqcup \bigsqcup_{i=1}^{\frac{m-1}{2}-1} \{2(i+1)m+2(2i+1)+1, 2(i+1)m+2(2i+1)+3, \dots, 2(i+2)m+3\}.
\end{split}
\]
}
An argument analogous to that used in the case of even $m$ shows that no element of ${\rm G}^1$ is the valuation of a polynomial in the $f_i$, which enables us to conclude that ${\rm S}_{\rm gen}={\rm S}^1$. As a result, ${\rm S}_{\rm gen}$ has genus
\[
g(m)= (2m-1)+ (m+2)+ \sum_{i=1}^{\frac{m-1}{2}}(m-2i-1)+ \sum_{i=1}^{\frac{m-1}{2}-1} (m-2i+1)= \frac{1}{2}m^2+ 2m-\frac{3}{2}.
\]
\end{proof}

\begin{coro}\label{codim_generic_Severi_n=3}
For every positive integer $m \geq 9$, the corresponding Severi variety $M^3_{d,g;{\rm S}_{\rm gen}}$ is of codimension strictly less than $g$ in $M^3_d$ whenever $d \gg g$.
\end{coro}

\begin{proof}
By genericity, the only conditions imposed on holomorphic maps $\mb{P}^1 \ra \mb{P}^3$ of degree $d$ by ${\rm S}_{\rm gen}$ arise from ramification in the preimage of the cusp; and there are $r_P-1= 6m-1$ of these. In light of Theorem~\ref{generic_semigroup_n=3}, it therefore suffices to show that
\[
6m-1 < \frac{1}{2}m^2+ 2m-\frac{3}{2}
\]
which is guaranteed by our hypothesis on $m$.
\end{proof}

\begin{rem}
\emph{We expect that the codimension estimates we have obtained for Severi varieties of rational unicuspidal curves, i.e., for linear series on $\mb{P}^1$ with unicuspidal images, hold more generally for (varieties of) linear series on {\it general} curves of arbitrary genus. Indeed, to establish the algebraic independence of conditions imposed by cusps on linear series on a given smooth curve $C$, it suffices to certify that certain evaluation maps from (sections of) line bundles to their associated jet bundles are surjective, and this is easy when $C$ is either rational or elliptic. On the other hand, when $C$ is {\it general}, it specializes in a flat family to a stable union $C_0$ of rational and elliptic curves; and it is possible to explicitly relate the variety of linear series on $C$ to the variety of {\it limit linear series} \cite{EH2} on $C_0$. To complete this argument, we would need a suitable generalization of the notion of ``cusp" for limit linear series.
}
\end{rem}

\end{document}